%% file: multiscale_analysis.tex
\documentclass[10pt, a4paper]{article}

\usepackage[utf8]{inputenc}
\usepackage{amsmath}
\usepackage{amssymb}
\usepackage{enumerate}
\usepackage{amsthm}

\newtheorem{theorem}{Theorem}[section]
\newtheorem{proposition}[theorem]{Proposition}
\newtheorem{lemma}[theorem]{Lemma}
\newtheorem{corollary}[theorem]{Corollary}
\newtheorem{remark}[theorem]{Remark}
\newtheorem{assumption}[theorem]{Assumption}

\newcommand{\R}{\mathbb{R}}
\newcommand{\N}{\mathbb{N}}

\newcommand{\norm}[1]{\left\| #1 \right\|}
\newcommand{\normsup}[1]{\left\|#1\right\|_{\infty}}
\newcommand{\normH}[1]{\left\|#1\right\|_{H^1(1+\rho)}}

\newcommand{\tw}{u^{TW}} 
\newcommand{\twp}{\hat{u}} 

\title{A Multiscale Analysis of Traveling Waves in Stochastic Neural Fields}

\author{Eva Lang\footnotemark[1]}

\begin{document}

\footnotetext[1]{Institut f\"ur Mathematik, Technische Universit\"at Berlin, D-10623 Berlin, Germany (lang@math.tu-berlin.de). The work was supported by the DFG RTG 1845.}

\maketitle

\begin{abstract}

We analyze the effects of noise on the traveling wave dynamics in neural fields. 
The noise influences the dynamics on two scales: first, it causes fluctuations in the wave profile, and second, it causes a random shift in the phase of the wave.  We formulate the problem in a weighted $L^2$-space, allowing us to separate the two spatial scales. By tracking the stochastic solution with a reference wave we obtain an expression for the stochastic phase. We derive an expansion of the stochastic wave, describing the influence of the noise to different orders of the noise strength. To first order of the noise strength, the phase shift is roughly diffusive and the fluctuations are given by a stationary Ornstein-Uhlenbeck process orthogonal to the direction of movement. This also expresses the stability of the wave under noise.

\end{abstract}

%

\pagestyle{myheadings}
\thispagestyle{plain}
\markboth{Eva Lang}{A Multiscale Analysis of Stochastic Traveling Waves}

\allowdisplaybreaks[2]

\input{introduction}

\input{setting}
\input{expansion}
\input{immediaterelaxation}

\input{asymptoticbehavior}

\bibliographystyle{siam}
\bibliography{C:/Users/Eva/ownCloud/bibliographies/generalreferences,C:/Users/Eva/ownCloud/bibliographies/neuromathreferences}
\end{document}

%% file: introduction.tex
\section{Introduction}

Since their introduction in \cite{amari} and \cite{wilsoncowan72}, \cite{wilsoncowan73}, neural field equations have been used widely as a phenomenological model for the evolution of the activity in networks of populations of neurons in the continuum limit. While they are of a relatively simple form, they exhibit a variety of interesting spatio-temporal patterns.
For an overview of the many interesting questions associated to and the analysis of neural fields we refer to the books \cite{ermentroutbook}, \cite{bressloffbook}, and \cite{coombesbook}, and to the topical reviews \cite{ermentroutreview}, \cite{bressloffreview}, and \cite{coombeswaves}.

We will here consider the one-dimensional Amari-type neural field equation
\begin{equation}
\label{eq:nfe}
\frac{\partial}{\partial t} u(x,t) = -u(x,t) + \int_{-\infty}^{\infty} w(x-y) F(u(y,t)) dy,
\end{equation}
where $u(x,t)$ is the average membrane potential in the population of neurons located at $x$ at time $t$, the kernel $w$ describes the strength of the synaptic connections, and the gain function $F$ relates the potential to the activity in the population. We will be interested in traveling wave solutions to this equation, which have been proven to exist in \cite{ermentroutmcleod}. 

The communication of neurons is subject to noise and it is therefore interesting to study stochastic versions of neural field equations. Since they have not been derived from single neuron models it is in particular not clear how noise translates from the single neuron level to the level of populations. We will here sum up all possible stochastic influence in an additive Gaussian noise term and consider the stochastic neural field equation
\begin{equation}
\label{eq:snfe}
du(x,t) = \Big(-u(x,t) + \int w(x-y) F(u(y,t)) dy \Big) dt + \epsilon  dW(x,t),
\end{equation}
where $W$ is a $Q$-Wiener process on a suitable function space, and $\epsilon$ is a small parameter describing the strength of the noise.

In this article, it is our main goal to provide a mathematically rigorous analysis of the influence of the noise on the traveling wave dynamics on multiple scales. The term multiscale refers mainly to two different spatial scales: first, shifts in the phase of the wave, that is, displacements of the wave profile from its uniformly translating position, and second, fluctuations in the wave profile. 
In order to determine an expression for the stochastic phase $C(t)$, we track the stochastic solution $u$ with a reference wave profile. Here the phase is defined as a minimizer of an $L^2$-distance between $u$ and all possible translations of the wave profile. Roughly, tracking $u$ is then achieved by dynamically adapting the speed of the reference wave $\hat{u}$ and thereby moving along the gradient of $\|u-\hat{u}(\cdot-C(t))\|$ towards the minimum. This can be seen as a dynamic version of the freezing of traveling waves applied by 
Lord and Thümmler \cite{lordthuemmler}.
We derive an expansion of $u$ of the form
\begin{align*}
	u(x,t) & = \twp(x-\varphi_k(t)) + \epsilon v_0(x,t) + \epsilon^2 v_1(x,t) + ... + \epsilon^k v_{k-1}(x,t) + r_{k}(x,t) \\
	\varphi_k(t) & = ct + \epsilon C_0(t) + \epsilon^2 C_1(t) + ... + \epsilon^k C_{k-1}(t),
\end{align*}
where the coefficients $v_k$ and $C_k$ are independent of $\epsilon$ and where the rest terms $r_k$ are of higher order in $\epsilon$. The term multiscale may thus also refer to the different orders of the noise strength, and the $v_k$ and $C_k$ describe the influence of the noise on the scale $\epsilon^{k+1}$. Here we have separated the two spatial scales: the $C_k$ describe the effects of the noise on the phase, and the $v_k$ on the wave profile. The expansion is valid up to a stopping time $\tau$ which can be shown to be large with high probability converging to $1$ as $\epsilon$ goes to $0$.

An analysis of the properties of the coefficients then allows to describe the effects of the noise. To first order of the noise strength, the phase shift, given by $C_0$, is roughly diffusive. Using the spectral properties of the linearized system we find that the fluctuations are to first order given by a stationary Ornstein-Uhlenbeck process that is orthogonal to the direction of movement, expressing in particular the stability of the traveling wave under the noise.

The question of how noise influences the traveling wave dynamics in neural fields was first considered by Bressloff and Webber in \cite{bressloffwebber}. They identified the two effects and then formally derived a decomposition of the solution. They obtained a stochastic differential equation as an approximate description of the shift of the phase to first order of the noise strength, and found that the noise causes diffusive wandering of the front. 

In \cite{kruegerstannat}, Kr\"{u}ger and Stannat made first steps towards a mathematically rigorous derivation of a decomposition of the solution and our results can be seen as an extension of their work.

The problem is also considered by Inglis and MacLaurin in \cite{inglismaclaurin}. Under assumptions on the spectral properties of the dynamics they analyze the local stability and long-time behavior of the stochastic solution. While we dynamically adapt the speed of the reference wave such that its distance to the stochastic wave becomes minimal, they derive a stochastic differential equation for the phase whose solution realizes the minimum exactly. 

The main novelty of our approach is that we work in a weighted $L^2$-space. 
The density $\rho$ that we introduce seems natural for the analysis in two aspects.
First, it allows us to separate the two spatial scales. The dynamics of the phase and of the fluctuations in the wave profile decouple, such that we can obtain a separate description of the effects of the noise on the two scales (cf. section \ref{subsection:expansion}).
Second, we can describe the spectral properties of the system in the space $L^2(\rho)$ (cf. section \ref{subsection:phaseofthewave}). This will allow us to derive stability properties of the dynamics.

The method of obtaining a description of the stochastic dynamics by separating the dynamics on different spatial (or temporal) scales and approximating to a certain order of $\epsilon$ is related to Bl\"{o}mker's work on amplitude equations (see for example \cite{bloemkercubic} or \cite{bloemkerquadratic}) .

It would also be interesting to consider the problem on larger time-scales. When adding external input to the neural field equation, interesting phenomena, such as stimulus-locking, occur. An investigation of the effects of noise in such situations has been started in \cite{bressloffwebber} and recently continued in \cite{bressloffkilpatrick}. It would be desirable to obtain a mathematically rigorous description, in particular taking care of the different time scales involved. 
We plan to address this question in future work.

Our methods do not rely on the particular neural field and traveling wave setting and it should be possible to extend the results to other types of equations and other types of patterns such as bumps or traveling pulses.

The article is structured as follows. In section \ref{section:setting} we outline the mathematical setting, explain our approach to determining the phase of the wave, and introduce our assumptions on the spectrum. The main results are presented in section \ref{section:multiscaleanalysis}, where we derive the expansion of the solution with respect to the noise strength (Thms. \ref{thm:expansion1}, \ref{thm:expansion2}, \ref{thm:immediaterelax}) and analyze properties of the coefficients (subsection \ref{subsection:immediaterelaxation}). Finally, in section \ref{section:asymptoticbehavior}, we show that our results apply in the case of exponential synaptic decay. In particular we analyze the asymptotic rates of decay of the derivative of the wave $\twp_x$ and the associated adjoint eigenfunction, which are given explicitly as the roots of a third degree polynomial.

%% file: setting.tex
\section{Setting}
\label{section:setting}

\subsection{Assumptions on the Parameters}
\label{subsection:parameters}

We denote by $\|\cdot\|$ and $\langle \cdot, \cdot \rangle$ the norm and scalar product in $L^2(\R)$, and for a density $\mu$ we denote by $\|\cdot\|_{\mu}$ and $\langle \cdot, \cdot \rangle_{\mu}$ the norm and scalar product in $L^2(\mu)$, the space of measurable functions $f$ such that $\|f\|_{\mu}^2 = \int f^2(x) \mu(x) dx< \infty$.
Let $H^1$ denote the Sobolev space of functions $h \in L^2(\R)$ that have a weak derivative $h_x \in L^2$ with norm $\|h\|_{H^1} = \big(\|h\|^2 + \|h_x\|^2\big)^{\frac{1}{2}}$.
The norm and scalar product in the Sobolev spaces $H^1$ and $H^1(\mu)$ are denoted by $\|\cdot\|_{H^1}$, $\langle \cdot, \cdot \rangle_{H^1}$ and $\|\cdot\|_{H^1(\mu)}$, $\langle \cdot, \cdot \rangle_{H^1(\mu)}$ , respectively.

As usual, we take the gain function $F:\R \rightarrow [0,1]$ to be a sigmoid function, for example $F(x) = \frac{1}{1+ e^{-\gamma(x-\kappa)}}$ for some $\gamma > 0$, $0<\kappa<1$. In particular we assume that
\begin{enumerate}[(i)]
	\item $F \geq 0, \lim_{x \downarrow -\infty}F(x)=0, \lim_{x\uparrow\infty}F(x)=1$
	\item $ F(x)-x$ has exactly three zeros $0 < a_1 < a < a_2<1$
	\item $F \in \mathcal{C}^k$ for some $k\geq 2$ and the derivatives are bounded
	\item  $F' > 0, F'(a_1) < 1, F'(a_2) < 1, F'(a)>1$
\end{enumerate}
Our assumptions on the synaptic kernel $w$ are the following
\begin{enumerate}[(i)]
	\item $w(x) = \bar{w}(|x|)$ for some $\bar{w} \in \mathcal{C}^1(\R_+, \R_+)$
	\item $\int_{-\infty}^{\infty} w(x) dx = 1$ and $\Big\|\frac{w_x}{w}\Big\|_{\infty} < \infty$
\end{enumerate} 
In Section \ref{section:asymptoticbehavior} on the asymptotic behavior of $\twp_x$ we assume exponential synaptic decay, i.e. $w(x) = \frac{1}{2\sigma} e^{-\frac{|x|}{\sigma}}$ for some $\sigma > 0$.

Assumption (iv) on $F$ implies that $a_1$ and $a_2$ are stable fixed points of (\ref{eq:nfe}), while $a$ is an unstable fixed point.
It has been shown in \cite{ermentroutmcleod} that under these assumptions there exists a unique monotone traveling wave solution to (\ref{eq:nfe}) connecting the stable fixed points (and in \cite{chenfengxin}, that traveling wave solutions are necessarily monotone). That is, there exists a unique wave profile $\hat{u}:\R \rightarrow [0,1]$ and a unique wave speed $c \in \R$ such that $u^{TW}(t,x) := \hat{u}(x-ct)$ is a solution to (\ref{eq:nfe}), i.e.
\begin{equation}
\label{eq:tw}
-c\partial_x \tw_t (x) = \partial_t \tw_t(x)  = -\tw_t(x) + w \ast F(\tw_t)(x),
\end{equation}
and 
\[\lim_{x \rightarrow -\infty} \hat{u}(x) = a_1, \qquad \lim_{x\rightarrow\infty} \hat{u}(x) = a_2.\]  
As also pointed out in \cite{ermentroutmcleod}, we can without loss of generality assume that $c\geq 0$.
Note that by (\ref{eq:tw}), the regularity of $F$ and $w$ determine that of $\twp$. If $F \in \mathcal{C}^k$, $\bar{w}\in \mathcal{C}^1$, then $\twp\in\mathcal{C}^{k+2}$.
Furthermore, $\hat{u}_x \in L^2(\R)$. 

As in \cite{kruegerstannat}, in order to solve (\ref{eq:snfe}), we will consider the stochastic neural field equation for the difference $v= u-\tw$ . In order to be able to `freeze' the wave (cf. subsection \ref{subsection:frozenwave}), we want that
$v(t)\in H^1$.
Let $(W_t)$ be a $Q$-Wiener process on $L^2$ with non-negative, symmetric covariance operator $Q$. We assume that $Q^{\frac{1}{2}}:L^2 \rightarrow H^1$ is Hilbert-Schmidt, that is, for any orthonormal basis $(e_k)$ of $L^2$, $\sum_k \| Qe_k \|_{H^1}^2 < \infty$. $Q^{\frac{1}{2}}$ may for example be given by a symmetric integral kernel, $Q^{\frac{1}{2}}u = \int q(x,y) u(y) dy$, with $\int \|q(x,\cdot)\|_{H^1}^2 dx < \infty$. 
Details on the theory of $Q$-Wiener processes can be found in \cite{daprato} or \cite{prevotroeckner}.

\begin{proposition}
\label{prop:exuni}
For any initial condition $\eta \in H^1$, there exists a unique strong $H^1$-valued solution 
to the stochastic neural field equation
\begin{equation}
\label{eq:snfedifference}
 \begin{split}
	dv(t) & = \big(-v(t) + w\ast(F(\twp(\cdot-ct)+v(t))-F(\twp(\cdot-ct)))\big) dt + \epsilon dW_t,\\
  v(0) & = \epsilon \eta.
\end{split}
\end{equation}
$v$ has a continuous modification and for all $p\geq 1$, $E(\sup_{0\leq t\leq T} \|v(t)\|_{H^1}^p)<\infty$. 
$u(x,t) = \twp(x-ct) + v(x,t)$ is a solution to (\ref{eq:snfe}).
\end{proposition}
\begin{proof}
$B(t,v) := -v + w\ast(F(\twp(\cdot-ct)+v)-F(\twp(\cdot-ct)))$ is Lipschitz continuous in $v$ since for $v_1, v_2 \in H^1$, 
\begin{align*}
	& \norm{w\ast(F(\twp(\cdot-ct)+v_1)-F(\twp(\cdot-ct)+v_2))}_{H^1} \\
	& \leq \Big( 1+ \normsup{\frac{w_x}{w}}\Big) \norm{w\ast(F(\twp(\cdot-ct)+v_1)-F(\twp(\cdot-ct)+v_2))}\\
	& \leq \Big( 1+ \normsup{\frac{w_x}{w}}\Big) \|F'\|_{\infty} \|v_1-v_2\|.
\end{align*}
Now the claim follows for example from Thm. 7.4 in \cite{daprato} (with $A=0$).
\qquad
\end{proof}

$v(t)$ satisfies
\[dv(t) = \big( L_t v(t) + R(t, v(t)) \big) dt + \epsilon dW_t,\]
where the time-dependent linear operator $L_t$ is given as
\begin{equation}
\label{eq:Lt}
L_t v = -v + w\ast (F'(\twp(\cdot-ct))v),
\end{equation}
and
\[R(t,v) = w\ast\big(F(\twp(\cdot-ct)+v)-F(\twp(\cdot-ct))-F'(\twp(\cdot-ct))v\big).\]

\subsection{The frozen wave equation}
\label{subsection:frozenwave}

It will be useful to work in the moving frame picture. That is, 
we can freeze the wave by moving instead the coordinates.
For $h:[0,T] \rightarrow H^1$ set $h^{\#}(x,t):= \Phi_t h(x,t) := h(x+ct, t)$. 
For $g \in H^1$ , by It\^{o}'s lemma,
\begin{align*}
	& \langle  v^{\#}(t), g\rangle  = \langle v(t), g(\cdot-ct)\rangle \\
	& = \langle v(0), g\rangle - c \int_0^t \langle v(s), g_x(\cdot-cs)\rangle ds \\
	& \qquad + \int_0^t \langle L_sv(s)+R(s,v(s)), g(\cdot-cs)\rangle ds  + \epsilon \int_0^t \langle g(\cdot-cs), dW_s\rangle \\
	& = \langle v^{\#}(0), g \rangle + c \int_0^t \langle \partial_x v^{\#}(s), g \rangle ds + \int_0^t \langle - v^{\#}(s) + w\ast(F'(\twp)v^{\#}(s)), g \rangle ds \\
	& \qquad {} + \int_0^t \langle R^{\#}(v^{\#}(s)),g\rangle ds  + \epsilon \int_0^t \langle g, \Phi_s dW_s \rangle \\
	& = \langle v^{\#}(0), g \rangle + \int_0^t \langle L^{\#} v^{\#}(s), g \rangle ds + \int_0^t \langle R^{\#}(v^{\#}(s)), g \rangle ds + \epsilon \int_0^t \langle g, \Phi_s dW_s \rangle ds,
\end{align*}
where $L^{\#}$ is the frozen wave operator given by
\begin{align*}
	L^{\#} v & = -v + c\partial_x v + w\ast(F'(\twp) v) , \\
	\mathcal{D}(L^{\#}) & = H^1,
\end{align*}
and $R^{\#}(v) = w\ast (F(\twp+v)-F(\twp)-F'(\twp)v)$.
That is, $v^{\#}(t)$ is the weak solution in $L^2$ to the frozen wave equation
\[ dv^{\#}(t) = \big(L^{\#} v^{\#}(t) + R^{\#}( v^{\#}(t)) \big) dt + \epsilon \Phi_t dW_t.\]

\noindent Note that $L^{\#}(\twp_x) = 0$. 
The adjoint of $L^{\#}$ is given as
\begin{align*}
	L^{\#,*}v &= -v - c\partial_x v + F'(\twp) w\ast v \\
	\mathcal{D}(L^{\#,*}) &= H^1.
\end{align*}

\begin{proposition}
There exists a unique $\psi \in H^1$, $\psi>0$, such that $L^{\#,*}\psi = 0$.
\end{proposition}
\begin{proof}
The proof is similar to that of Thm. 4.2 and 4.3 in \cite{ermentroutmcleod}.
There exist $\delta, M>0$ such that for all $x$ with $|x|\geq M$, $F'(\twp(x)) \leq 1-\delta$.
Consider the operator on $L^2$ \[w\ast(F'(\twp)v) = Kv + Bv \]
where \[Kv(x) = \int_{-M}^{M} w(x-y) F'(\twp(y)) v(y) dy\]
and \[Bv(x) = \bigg( \int_{-\infty}^{-M} + \int_M^{\infty} \bigg) w(x-y) F'(\twp(y)) v(y) dy.\]
Then $\|B\| \leq 1-\delta$ and $K$ is compact.
We have \[L^{\#} - (-I+c\partial_x +B) = K.\]
$A:=-I+c\partial_x + B$ has a bounded inverse.

We have $A^{-1}L^{\#} = I + A^{-1}K$ and $A^{-1}K$ is compact. Thus, $A^{-1}L^{\#}$ is a Fredholm operator with index $0$.
By the Fredholm alternative, since $A^{-1}L^{\#}\twp_x = 0$ and since there are no other eigenfunctions with eigenvalue $0$ as proven in Thm. 4.2 in \cite{ermentroutmcleod}, there exists a unique $\tilde{\psi} \not\equiv 0$ such that $L^{\#,*}(A^*)^{-1}\tilde{\psi}=0$,
hence $L^{\#,*}\psi=0$ where $\psi := (A^*)^{-1} \tilde{\psi} \in H^1$. 

Since $c\psi_x = -\psi + F'(\twp) w\ast\psi$, we actually have $\psi \in \mathcal{C}^1$. 
We show that $\psi$ is of one sign. Assume without loss of generality that there exists $x$ such that $\psi(x) > 0$. Set $\psi^+(x) = \psi(x) \vee 0$.
Then $\psi^+\in H^1$ with $\psi^+_x \equiv \psi_x$ on $\{\psi \geq 0\}$ and $\psi^+_x \equiv 0$ on $\{\psi < 0\}$.
Thus, on $\{\psi \geq 0\}$,  $L^{\#,*}\psi^+ = -\psi -c\psi_x + F'(\twp)w\ast\psi^+ \geq L^{\#,*}\psi(x) = 0$, and on $\{\psi<0\}$, $L^{\#,*}\psi^+ = F'(\twp)w\ast\psi^+ \geq 0$.
Since $\twp_x>0$ and
\[0 = \langle L^{\#} \twp_x, \psi^+ \rangle = \langle \twp_x, L^{\#,*}\psi^+\rangle, \]
it follows that $L^{\#,*}\psi^+ \equiv 0$ and hence $\psi^+ \equiv \psi$. 
\qquad
\end{proof}

We normalize $\psi$ such that $\langle \twp_x, \psi \rangle = 1$.
Set $\rho(x) = \frac{\psi(x)}{\twp_x(x)}$.
Note that for $h \in H^1$,
\begin{equation}
\label{eq:projectionproperty}
\langle L^{\#} h, \twp_x \rangle_{\rho} = \langle L^{\#}h, \psi\rangle = \langle h, L^{\#,*} \psi \rangle = 0,
\end{equation}
that is, $L^{\#}(H^1) \subset \twp_x^{\perp}$, where we denote by $\twp_x^{\perp}$ the orthogonal complement of $\twp_x$ in $L^2(\rho)$.
In $L^2(\rho)$, the direction of movement of the wave $\twp_x$ and the orthogonal directions are thus naturally separated by $L^{\#}$, which makes it a natural choice of function space to work in.

\subsection{The phase of the wave}
\label{subsection:phaseofthewave}
We loosely define the phase $\varphi$ of a `wave-like' function $u$ to be a minimizer of 
\begin{equation}
\label{eq:minphase}
a \mapsto \|u-\twp(\cdot-a)\|_{\rho}, 
\end{equation}
the $L^2(\rho)$-distance between $u$ and all possible translations of the deterministic wave profile $\twp$.

In order to determine the phase shift caused by the noise, we dynamically adapt the phase of a reference wave to match that of the stochastic solution. The idea is to move along the gradient of (\ref{eq:minphase}) towards the minimum.
If we let $a$ depend on a parameter $s$ and differentiate, we obtain
\[\frac{d}{ds} \norm{u-\twp(\cdot-a(s))}_{\rho}^2 = 2 \dot{a}(s) \langle u-\twp(\cdot-a(s)), \twp_x(\cdot-a(s))\rangle_{\rho}.\]
If we now choose $a$ such that $\dot{a}(s) = -\langle  u-\twp(\cdot-a(s)), \twp_x(\cdot-a(s)) \rangle_{\rho}$,\\ then $\frac{d}{ds}  \norm{u-\twp(\cdot-a(s))}_{\rho}^2 \leq 0$, which means $a$ should move towards the right phase.

This motivates the following dynamics which were first introduced in \cite{stannatnagumo} and \cite{stannatreactiondiffusion}.
Let $C^m(t)$ be the solution to the pathwise ordinary differential equation
\begin{equation}
\label{eq:adaptationwavespeed}
\begin{split}
 c^m(t) 
  &:= \dot{C}^m(t)\\
	&: = -m \langle u(t)- \twp(\cdot-ct-C^m(t)), \twp_x(\cdot-ct-C^m(t)) \rangle_{\rho(\cdot-ct-C^m(t))} \\
	& = -m\langle u(t)-\twp(\cdot-ct-C^m(t)), \psi(\cdot-ct-C^m(t)) \rangle,
\end{split}
\end{equation}
which can be shown to exist analogously to \cite{kruegerstannat}, Prop. 3.5.
Here $m>0$ is a parameter that determines the rate of relaxation to the right phase. 

It cannot in general be expected that there exists a unique global minimum of (\ref{eq:minphase}) as discussed in \cite{inglismaclaurin}. Here $C^m$ is designed to follow the local minimum that is closest to the initial phase.

In \cite{inglismaclaurin}, Inglis and MacLaurin derive an SDE describing the dynamics of this local minimum exactly.
Our approach gives an approximate description in terms of an ODE. In particular it provides a way of calculating the phase of the stochastic wave from a realization without explicit knowledge of the noise.

Recall that $\twp_x \in H^1(\rho)$ is an eigenfunction with eigenvalue $0$ of the frozen wave operator $L^{\#}$. We assume that the neural field dynamics is contractive on $\twp_x^{\perp}$.
\begin{assumption}
\label{ass:spectralgap}
$L^{\#}$ has a spectral gap in $L^2(\rho)$, that is, there exists $\kappa > 0$ such that for $h\in H^1(\rho)$,
\begin{equation}
\label{eq:spectralgap}
\langle L^{\#} h, h\rangle_{\rho} \leq - \kappa \big( \norm{h}_{\rho}^2 - \langle h, \twp_x \rangle_{\rho}^2\big).
\end{equation}
\end{assumption}
This assumption reflects that perturbations in the front profile that are orthogonal to the direction of movement $\twp_x$ should be damped by the neural field dynamics, while perturbations in the direction of movement will cause a random shift in the phase of the wave. 
An article containing the proof is in preparation.

Note that in contrast to \cite{kruegerstannat} and \cite{inglismaclaurin}, here we assume the spectral gap in $L^2(\rho)$ instead of $L^2$. To our knowledge, there is so far no proof of the $L^2$-version for arbitrary values of the wave speed $c$.

Under Assumption \ref{ass:spectralgap}, using (\ref{eq:projectionproperty}), $L^{\#}$ generates a contraction semigroup $(P^{\#}_t)$ on $\twp_x^{\perp}$ that satisfies 
\begin{equation}
\label{eq:contractionsg} 
\big\|P^{\#}_t h\big\|_{\rho} \leq e^{-\kappa t}\norm{h}_{\rho}.
\end{equation}

%% file: expansion.tex
\section{A Multiscale Analysis}
\label{section:multiscaleanalysis}

\input{measurerho}

\subsection{An SDE for the wave speed}

Set $v^m(x,t) = u(x,t)-\twp(x-ct-C^m(t))$ to be the difference between the solution $u$ to the stochastic neural field equation (\ref{eq:snfe}) and the deterministic wave profile moving at the dynamically adapted speed $c+ c^m(t)$ as defined in (\ref{eq:adaptationwavespeed}).
$v^m$ satisfies the stochastic evolution equation
\begin{equation}
\label{eq:vm}
\begin{split}
	dv^m(t) 
	& = \Big( -v^m(t) + w\ast (F'(\twp(\cdot-ct-C^m(t))) v^m(t)) + R^m(t,v^m(t)) \\
	& \qquad {} +c^m(t) \twp_x(\cdot-ct-C^m(t))  \Big) dt + \epsilon dW_t,
\end{split}
\end{equation}
where
\begin{align*}
R^m(t, v^m(t)) 
& =  w\ast\Big( F(\twp(\cdot-ct-C^m(t))+v^m(t))-F(\twp(\cdot-ct-C^m(t)))\\
& \qquad {} -F'(\twp(\cdot-ct-C^m(t)))v^m(t)\Big).
\end{align*}
Set $\varphi^m(t) = ct+C^m(t)$.
\begin{lemma}
\label{lemma:sdewavespeed}
The adaptation of the wave speed $c^m(t)=-m\langle v^m(t), \psi(\cdot-\varphi^m(t))\rangle$ solves the SDE
\begin{align*}
dc^m(t) 
& = \Big( -mc^m(t) + m \langle  v^m(t), \psi_x(\cdot- \varphi^m(t)) \rangle c^m(t)  \\
& \qquad {} - m \langle R^m(t, v^m(t)), \psi(\cdot- \varphi^m(t)) \rangle \Big) dt  -\epsilon m \langle \psi(\cdot- \varphi^m(t)), dW_t \rangle, \\
c^m(0) &= -\epsilon m \langle \eta, \psi\rangle.
\end{align*}
\end{lemma}
\begin{proof}
By It\^{o}'s lemma, 
\begin{align*}
	 c^m(t)  
	& = -\epsilon m \langle \eta, \psi \rangle  + m \int_0^t (c + c^m(s))\langle v^m(s), \psi_x(\cdot-\varphi^m(s)) \rangle ds \\
	& \qquad {} -m \int_0^t \langle - v^m(s)+ w\ast\Big(F'(\twp(\cdot- \varphi^m(s)))v^m(s)\Big) , \psi(\cdot-\varphi^m(s)) \rangle ds \\
	& \qquad {} - m \int_0^t \langle R^m(s, v^m(s)), \psi(\cdot-\varphi^m(s))\rangle ds - m \int_0^t c^m(s) ds \langle \twp_x, \psi \rangle \\
	& \qquad {}- \epsilon m \int_0^t \langle \psi(\cdot-\varphi^m(s)), dW_s\rangle \\
	& = c^m(0) -m \int_0^t \langle v^m(s), L^{\#,*}\psi (\cdot- \varphi^m(s)) \rangle ds  \\
	& \qquad {} - m \int_0^t \langle R^m(s,v^m(s)), \psi(\cdot-\varphi^m(s))\rangle ds \\
	& \qquad {} + m  \int_0^t c^m(s) \langle v^m(s),  \psi_x(\cdot-\varphi^m(s)) \rangle ds -m \int_0^t c^m(s) ds \\
	& \qquad {} -\epsilon m \int_0^t \langle \psi(\cdot-\varphi^m(s)),  dW_s \rangle \\
	& =  c^m(0) - m \int_0^t c^m(s) ds  + m  \int_0^t c^m(s) \langle v^m(s), \psi_x (\cdot-\varphi^m(s)) \rangle ds \\
	& \qquad {} - m \int_0^t \langle R^m(s, v^m(s)), \psi(\cdot-\varphi^m(s))\rangle ds   - \epsilon m \int_0^t \langle \psi(\cdot-\varphi^m(s)),  dW_s \rangle 
\end{align*}
\end{proof}

\subsection{Expansion with respect to the noise strength}
\label{subsection:expansion}

As outlined in section \ref{subsection:phaseofthewave} we expect $\twp(\cdot-ct-C^m(t))$ to track the stochastic solution, which means $v^m$ should describe the fluctuations in the wave profile. As long as $m$ is finite, this can however only be an approximate description.

We prove an expansion of the solution $u$ to (\ref{eq:snfe}) that allows to analyze the behavior of the coupled system $(v^m, C^m)$ to arbitrary order of $\epsilon$. In Section \ref{subsection:immediaterelaxation} we will derive the expansion in the limit $m\rightarrow\infty$ and analyze properties of the coefficients in the expansion. In particular, we will show that the limiting regime indeed corresponds to immediate relaxation to the right phase, thereby justifying the expansion as a description of the effects of the noise.

Set $\rho_t(x) = \rho(x-ct)$.
For $h\in \mathcal{C}([0,T], H^1(1+\rho))$ set 
\[\|h\|_T = \sup_{0\leq t\leq T} \|h(t)\|_{H^1(1+\rho_t)},\]
 and for $f\in C([0,T])$ set $|f|_T = \sup_{0 \leq t\leq T} |f(t)|$.
Here we move the measure with the wave such that for all $t\geq 0$,
\[\norm{\partial_x u^{TW}_t}_{H^1(1+\rho_t)}= \norm{\twp_x(\cdot-ct)}_{H^1(1+\rho_t)} = \norm{\twp_x}_{H^1(1+\rho)}.\]

Note that there exists a constant $K>0$ such that $\normsup{h} \leq K \norm{h}_{H^1(1+\rho)}$ for all $h \in H^1(1+\rho)$.

We start by formally identifying the highest order terms in $c^m(t)$ using Lemma \ref{lemma:sdewavespeed}. Since we expect both $C^m$ and $v^m$ to be of order $\epsilon$ (up to the time horizon $T$) that leads us to define $c^m_0(t)$ to be the unique strong solution to
\begin{equation}
\label{eq:cm0}
dc^m_0(t) = -mc^m_0(t)dt - m \langle \psi(\cdot-ct), dW_t \rangle, \qquad c^m_0(0) = -m\langle \eta, \psi \rangle.
\end{equation}
Set $C^m_0(t) = \int_0^t c^m_0(s) ds$ and $\varphi^m_0(t) = ct + \epsilon C^m_0(t)$.

Formally identifying the highest order terms in (\ref{eq:vm}) we define $v^m_0$ to be the unique strong solution to
\[dv^m_0(t) = \big(L_t v^m_0(t) + c^m_0(t) \twp_x(\cdot-ct) \big) dt + dW_t, \qquad v^m_0(0) = \eta,\]
where $L_t$ is as defined in (\ref{eq:Lt}).

\begin{remark}
Note that, to first order, the dynamics of $C^m$ decouple from those of $v^m$. This would not be the case if we had defined the phase in the unweighted space $L^2$. It is by defining the phase adaptation in $L^2(\rho)$ that we can achieve a separate description of the influence of the noise on the two scales.
\end{remark}

For $\epsilon>0$, $0\leq q < 1$, set
\begin{equation}
\label{eq:tau}
\tau_{q, \epsilon} = \inf \{ 0\leq t \leq T: \|v(t)\|_{H^1(1+\rho_t)} \geq \epsilon^{1-q}\},
\end{equation}
where $v$ is the solution from Prop. \ref{prop:exunirho}, and
\[\tau_{q, \epsilon}^{m} = \inf \{ 0 \leq t\leq T: \ |C^m_0(t)|\geq \epsilon^{-q}\}.\]

\begin{theorem}
\label{thm:expansion1}
Let $q<\frac{1}{2}$. 
Then on $\{\tau_{q,\epsilon} \wedge \tau_{q,\epsilon}^{m} = T\}$,
\[u(x,t) = \twp(x-ct-\epsilon C^m_0(t)) + \epsilon  v^m_0(t) + \epsilon r_1^m(t)\]
with
\[\|r^m_1\|_T \leq \alpha_1(T) \epsilon^{1-2q}\]
for a constant $\alpha_1(T)$ independent of $\epsilon$ and $m$,
and
\[P(\tau_{q,\epsilon} \wedge \tau^{m}_{q,\epsilon} =T) \xrightarrow{\epsilon\rightarrow 0} 1.\]
\end{theorem}

\begin{proof}
Set $\tilde{v}^m_0(t) = u(t)-\twp(\cdot-\varphi^m_0(t)) =v(t) + \twp(\cdot-ct)-\twp(\cdot-\varphi^m_0(t))$. 
Note that by Taylor's formula, there exists $\xi(x,t)$ with $|\xi|\leq \epsilon |C^m_0|$ such that $\tilde{v}^m_0(x,t) = v(x,t) + \epsilon C^m_0(t) \twp_x(x-ct+\xi(x,t))$ and thus
\[\|\tilde{v}^m_0(t)\|_{H^1(1+\rho_t)} \leq \|v(t)\|_{H^1(1+\rho_t)} + \epsilon|C^m_0(t)| \|\twp_x\|_{H^1(1+\rho(\cdot-\xi(t))}. \]
Using (\ref{eq:rho1}) and (\ref{eq:boundrhox}), 
\begin{equation}
\label{eq:rho3}
\rho(x-\xi)\leq (L_{\rho} \vee e^{M\xi}) \rho(x),
\end{equation}
and thus on $\{\tau_{q,\epsilon}\wedge\tau_{q,\epsilon}^{m}=T\}$,
\begin{equation}
\label{eq:tildevm0}
\|\tilde{v}^m_0\|_T \leq \epsilon^{1-q} \big(1+ (L_{\rho} \vee e^{M\epsilon^{1-q}})^{\frac{1}{2}} \|\twp_x\|_{H^1(1+\rho)}\big)
\end{equation}

$r^m_1$ satisfies the pathwise evolution equation
\begin{align*}
dr^m_1(t) 
& = \Big( L_t r^m_1(t)  \\
& \quad + \frac{1}{\epsilon} w\ast \big( F(\twp(\cdot - \varphi^m_0(t)) + \tilde{v}^m_0(t) )-F(\twp(\cdot-\varphi^m_0(t)))\\
& \qquad - F'(\twp(\cdot-\varphi^m_0(t)))\tilde{v}^m_0(t)\big) \\
& \quad + \frac{1}{\epsilon} w\ast \big( ( F'(\twp(\cdot-\varphi^m_0(t)))-F'(\twp(\cdot-ct))) \tilde{v}^m_0(t)\big) \\
& \quad + c^m_0(t) (\twp_x(\cdot-\varphi^m_0(t))-\twp_x(\cdot-ct)) \Big) dt \\
& =: \big( L_t r^m_1(t) + r^m_{1,1}(t) + r^m_{1,2}(t) + r^m_{1,3}(t) \big) dt.
\end{align*}

\noindent By Taylor's theorem there exist $\xi_{1,1}(x,t)$, $\xi_{1,2}(x,t)$ such that
\begin{align*}
\epsilon r_{1,1}^m(t) & = \frac{1}{2} w\ast \big(F''(\twp(\cdot-\varphi^m_0(t))+\xi_{1,1}(t))(\tilde{v}^m_0(t))^2\big),\\
\epsilon r_{1,2}^m(t) & = - \epsilon C^m_0(t) w\ast \big(F''(\twp(\cdot-ct+\xi_{1,2}(t)))\twp_x(\cdot-ct+\xi_{1,2}(t))\tilde{v}^m_0(t)\big).
\end{align*}
We therefore have, using the Cauchy-Schwarz inequality, that
\begin{align*}
\|r_{1,1}^m(t)\|_{1+\rho_t}^2
& \leq \frac{1}{4\epsilon^2}  \norm{F''}_{\infty}^2 \int \int w^2(x-y) (\tilde{v}^m_0)^2(y,t)dy (1+\rho_t(x)) dx\\
& \qquad \int (\tilde{v}^m_0)^2(y,t) dy  \\
& \leq \frac{1}{4\epsilon^2}  \norm{F''}_{\infty}^2 \|w\|_{\infty}(1+ K_{\rho} ) \|\tilde{v}^m_0(t)\|^2 \|\tilde{v}^m_0(t)\|_{1+\rho_t}^2
\end{align*}
and
\begin{align*}
\|r_{1,2}^m(t)\|_{1+\rho_t}^2 
& \leq |C^m_0(t)|^2 \|F''(\twp)\twp_x\|_{\infty}^2 (1+K_\rho ) \|\tilde{v}^m_0(t)\|_{1+\rho_t}^2.
\end{align*}

\noindent Recall that $r^m_1$ can be represented as a mild solution,
\[r^m_1(t) = \int_0^t P_{t,s} \big( r^m_{1,1}(s)+r^m_{1,2}(s) + r^m_{1,3}(s)\big) ds.\]
Set $R^m_{1,3}(t) = \frac{1}{\epsilon} \big(- \twp(\cdot-\varphi^m_0(t))+\twp(\cdot-ct)-\epsilon C^m_0(t)\twp_x(\cdot-ct)\big)$.
We have $r^m_{1,3}(t) = \big(\frac{d}{dt}+c\partial_x) R^m_{1,3}(t)$ and therefore
\begin{align*}
\int_0^t P_{t,s}r^m_{1,3}(s)ds 
& = \int_0^t \frac{d}{ds} \big[ P_{t,s}R^m_{1,3}(s)\big] + P_{t,s}(L_s+c\partial_x)R^m_{1,3}(s) ds \\
& = R^m_{1,3}(t) + \int_0^t P_{t,s}(L_s+c\partial_x )R^m_{1,3}(s) ds .
\end{align*}
Recall that $\|P_{t,s}h\|_{H^1(1+\rho_s)} \leq e^{L_*(t-s)} \|h\|_{H^1(1+\rho_s)}$. Using Taylor's theorem and (\ref{eq:rho3}) it follows that 
\begin{align*}
& \bigg\|\int_0^t P_{t,s}r^m_{1,3}(s) ds\bigg\|_{H^1(1+\rho_t)} \\
& \leq \frac{\epsilon}{2} |C^m_0(t)|^2 \|\twp_{xx}(\cdot-ct-\xi_{1,3,1}(t))\|_{H^1(1+\rho_t)} \\
& \quad {} + L_{\rho}^{\frac{1}{2}} \int_0^t e^{L_*(t-s)} \frac{\epsilon}{2} |C^m_0(s)|^2 \big( L_*  \|\twp_{xx}(\cdot-cs-\xi_{1,3,1}(s))\|_{H^1(1+\rho_s)} \\
& \quad {} + c \|\twp_{xxx}(\cdot-cs-\xi_{1,3,2}(s))\|_{H^1(1+\rho_s)} \big) ds \\
& \leq \frac{\epsilon}{2}  |C^m_0|_T^2 (L_{\rho} \vee e^{M\epsilon^{1-q}})^{\frac{1}{2}} \Big( (1+L_{\rho}^{\frac{1}{2}} (e^{L_*T}-1)) \|\twp_{xx}\|_{H^1(1+\rho)}\\
& \quad {} + \frac{cL_{\rho}^{\frac{1}{2}}}{L_*}  (e^{L_*T}-1)\|\twp_{xxx}\|_{H^1(1+\rho)}\Big).
\end{align*}
Since for $i=1,2$
\begin{align*}
\|P_{t,s} r^m_{1,i}(s)\|_{H^1(1+\rho_t)}^2 
& \leq L_{\rho} e^{2L_*(t-s)} \|r_{1,i}^m(s)\|^2_{H^1(1+\rho_s)} \\
& \leq  L_{\rho}  e^{2L_*(t-s)} \Big(1+\Big\| \frac{w_x}{w} \Big\|^2_{\infty} \Big)\|r^m_{1,i}(s)\|_{1+\rho_s}^2,
\end{align*}
we conclude that there exists a constant $\alpha_1(T)$ independent of $m$ and $\epsilon$ such that
\[\|r^m_1\|_T \leq \alpha_1(T) \epsilon^{1-2q}.\]

If $\tau_{q,\epsilon}\wedge \tau_{q,\epsilon}^{m} = \tau_{q,\epsilon} < T$, then by continuity, almost surely, there exists $t_0<T$ such that, 
\begin{align*}
\epsilon^{1-q} 
& = \|v(t_0)\|_{H^1(1+\rho_{t_0})} \\
& = \| - \epsilon C^m_0(t_0) \twp_x(\cdot-ct_0+\xi(t_0)) + \epsilon v^m_0(t_0) +\epsilon r^m_1(t_0)\|_{H^1(1+\rho_{t_0})}
\end{align*}
and thus
\[\|v_0^m(t_0) - C^m_0(t_0) \twp_x(\cdot-ct_0+\xi(t_0))\|_{H^1(1+\rho_{t_0})} \geq \epsilon^{-q} -\epsilon \|r^m_1(t_0)\|_{H^1(1+\rho_{t_0})}.\]
We therefore have that
\begin{align*}
& P(\tau_{q,\epsilon} \wedge \tau_{q,\epsilon}^m = \tau_{q,\epsilon} < T)\\
& \leq P\big(\|v^m_0 - C^m_0 \twp_x(\cdot-ct+\xi(t))\|_T  \geq \epsilon^{-q} - \alpha_1(T) \epsilon^{1-2q}\big) \\
& \leq \frac{2 \epsilon^{2q}}{(1-\alpha_1(T)\epsilon^{1-q})^2}  \big( E(\|v^m_0\|_T^2) + E(|C^m_0|_T^2) (L_{\rho}\vee e^{M\epsilon^{1-q}})\|\twp_x\|_{H^1(1+\rho)}^2\big)  \\
& \xrightarrow{\epsilon\rightarrow 0} 0.
\end{align*}
Since
\[P(\tau^{m}_{q,\epsilon} < T) \leq P(|C^m_0|_T \geq \epsilon^{-q}) \leq \epsilon^{2q} E(|C^m_0|_T^2) \xrightarrow{\epsilon\rightarrow 0} 0,\]
it follows that
\[P(\tau_{q,\epsilon} \wedge \tau_{q,\epsilon}^{m} < T) \leq P(\tau_{q,\epsilon} \wedge \tau_{q,\epsilon}^{m}  = \tau_{q,\epsilon} < T) + P(\tau^{m}_{q,\epsilon} < T) \xrightarrow{\epsilon\rightarrow 0} 0. \]
\end{proof}

Analogously we can obtain an expansion to higher order of $\epsilon$.
Formally identifying the highest order terms in $\frac{1}{\epsilon} c^m(t)- c^m_0(t)$ we define $c^m_1(t)$ to be the unique strong solution to 
\begin{equation}
\label{eq:sdec1m}
\begin{split}
	dc^m_1(t) & = \big( -mc^m_1(t) - \frac{1}{2} m \langle w\ast(F''(\twp(\cdot-ct))(v^m_0)^2(t)), \psi(\cdot-ct) \rangle\\
	& \qquad {} + m c^m_0(t) \langle v^m_0(t), \psi_x(\cdot-ct) \rangle \big) dt  + m C^m_0(t) \langle \psi_x(\cdot-ct), dW_t\rangle, \\
	c^m_1(0) & = 0.
\end{split}
\end{equation}
Set $C^m_1(t) = \int_0^t c^m_1(s) ds$ and $\varphi^m_1(t) = ct + \epsilon C^m_0(t)+ \epsilon^2 C^m_1(t)$.
Identifying the highest order terms in $\frac{1}{\epsilon}v^m-  v^m_0$, set $v^m_1$ to be the unique strong solution to
\begin{equation}
\label{eq:sdev1m}
\begin{split}
	dv^m_1(t) 
	& = \Big(L_tv^m_1(t) + w\ast \Big(F''(\twp(\cdot-ct))\big(\frac{1}{2}(v^m_0)^2(t)-C^m_0(t)\twp_x(\cdot-ct) v^m_0(t)\big)\Big) \\
	& \qquad {}- c^m_0(t) C^m_0(t) \twp_{xx}(\cdot-ct)  + c^m_1(t) \twp_x(\cdot-ct)  \Big)dt,\\
	v^m_1(0) & = 0.
\end{split}
\end{equation}
For arbitrary $k \in \N$, if $F \in \mathcal{C}^{k+1}$ we can iterate the procedure and define $c^m_{k-1}$, and $v^m_{k-1}$ by successively identifying the highest order terms in $\frac{1}{\epsilon^{k-1}} (c^m-\epsilon c^m_0 - \ldots - \epsilon^{k-1}c^m_{k-2})$ and $ \frac{1}{\epsilon^{k-1}}(\tilde{v}^m(t)-\epsilon v^m_0(t) - \ldots - \epsilon^{k-1} v^m_{k-2}(t)$. This way we obtain an expansion of $u$ up to order $\epsilon^k$.
Set \[\varphi_{k-1}^m(t) = ct + \epsilon C^m_0(t) + \ldots + \epsilon^k C^m_{k-1}(t).\]
\begin{theorem}
\label{thm:expansion2}
Assume that $F \in \mathcal{C}^{k+1}$ for some $k\geq 1$. Let $q< \frac{1}{k+1}$. 
Then on $\{ \tau_{q,\epsilon} \wedge \tau_{q,\epsilon}^m=T\}$,
\[u(x,t) = \twp(\cdot-\varphi^m_{k-1}(t)) + \epsilon v^m_0(t) + \ldots + \epsilon^{k} v^m_{k-1}(t) + \epsilon^k r^m_k(t)\]
with
\[\|r^m_k\|_T \leq \alpha_k(T) \epsilon^{1-(k+1)q}\]
for some constant $\alpha_k(T)$ independent of $\epsilon$ and $m$.
\end{theorem}
We will prove the expansion for $k=2$. For larger $k$ the analysis can be carried out analogously, but the formulas become unwieldy.

We start by deriving a useful representation of $C^m_1$.
\begin{lemma}
\label{lemma:C1m}
\begin{align*}
C_1^m(t) 
& = -  \int_0^t (1-e^{-m(t-s)})\langle w\ast\Big(F''(\twp(\cdot-cs))\big(\frac{1}{2} (v^m_0)^2(s) \\
& \qquad {} -C^m_0(s) \twp_x(\cdot-cs) v_0^m(s)\big)\Big), \psi(\cdot-cs)\rangle  ds\\
& \quad + \int_0^t m e^{-m(t-s)} C^m_0(s) \langle v^m_0(s), \psi_x(\cdot-cs) \rangle ds\\
& \quad - \frac{1}{2}\int_0^t me^{-m(t-s)} (C^m_0)^2(s) ds \langle \psi_x, \twp_x\rangle.
\end{align*}
\end{lemma}
{\em Proof}.
Integrating (\ref{eq:sdec1m}) we obtain
\begin{align*}
C^m_1(t)
& = \int_0^t (1-e^{-m(t-s)}) \Big(- \frac{1}{2} \langle w\ast(F''(\twp(\cdot-cs))(v^m_0)^2(s)), \psi(\cdot-cs)\rangle \\ 
& \quad +  c^m_0(s)\langle v^m_0(s),\psi_x(\cdot-cs)\rangle \Big)  ds \\
& \quad + \int_0^t (1-e^{-m(t-s)}) C^m_0(s) \langle \psi_x(\cdot-cs), dW_s\rangle.
\end{align*}
By It\^{o}'s Lemma,
\begin{align*}
& \int_0^t (1-e^{-m(t-s)}) c^m_0(s) \langle v^m_0(s), \psi_x(\cdot-cs) \rangle ds \\
& = \int_0^t m e^{-m(t-s)} C^m_0(s) \langle v^m_0(s), \psi_x(\cdot-cs) \rangle ds  \\
& \quad - \int_0^t (1-e^{-m(t-s)})  C^m_0(s)\Big( \langle v^m_0(s), (L_s^*-c\partial_x) \psi_x(\cdot-cs)\rangle \\
& \qquad + c^m_0(s) \langle \twp_x,\psi_x \rangle \Big) ds \\
& \quad - \int_0^t (1-e^{-m(t-s)}) C^m_0(s) \langle \psi_x(\cdot-cs), dW_s\rangle.
\end{align*}
Using integration by parts we obtain that
\[- \int_0^t (1-e^{-m(t-s)}) c^m_0(s) C^m_0(s) ds = \frac{1}{2}\int_0^t - me^{-m(t-s)} (C^m_0)^2(s) ds,\]
and since $ (L_s^*-c\partial_x) \psi_x(\cdot-cs) = - F''(\twp(\cdot-cs))\twp_x(\cdot-cs)w\ast\psi(\cdot-cs) $, the claim follows.
\qquad\endproof

\begin{proof}[ of Thm. \ref{thm:expansion2}]

Note first that, using Lemma \ref{lemma:C1m}, 
\begin{align*}
|C^m_1|_T 
& \leq  \|\psi\| \int \Big\|w\ast\Big(F''(\twp(\cdot-cs))\big(\frac{1}{2}(v^m_0(s))^2 -C^m_0(s) \twp_x(\cdot-cs) v^m_0(s)\big)\Big)\Big\| ds \\
& \quad + |C^m_0|_T \|\psi_x\| \int_0^t m e^{-m(t-s)} \|v^m_0(s)\| ds + \frac{1}{2} |C^m_0|_T^2 .
\end{align*}
Since 
\[v^m_0(t) = \frac{1}{\epsilon} \big(v(t) + \twp(\cdot-ct)-\twp(\cdot-\varphi^m_0(t))\big) -  r^m_1(t) = \frac{1}{\epsilon} \tilde{v}^m_0(t)-r^m_1(t),\]
using Theorem \ref{thm:expansion1} and (\ref{eq:tildevm0}) it follows that there exists a constant $\beta_1(T)$ such that on $\{\tau_{q,\epsilon}\wedge \tau_{q,\epsilon}^m =T\}$
\[\|v^m_0\|_T \leq \beta_1(T) \epsilon^{-q}.\]
Therefore there exists a constant $\beta_2(T)$ such that
\begin{equation}
\label{eq:Cm1estimate}
|C^m_1|_T \leq \beta_2(T) \epsilon^{-2q}.
\end{equation}
Set $\tilde{v}^m_1(t) = v + \twp(\cdot-ct)-\twp(\cdot-\varphi^m_1(t))$.
By Taylor's theorem there exists $\xi(x,t)$ with $|\xi| \leq |\epsilon C^m_0 + \epsilon^2 C^m_1|$ such that
\[\tilde{v}^m_1(t) = v(t) + (\epsilon C^m_0(t)+ \epsilon^2 C^m_1(t)) \twp_x(\cdot+\xi(t))\]
and it follows that there exists a constant $\beta_3(T)$ such that
\[\|\tilde{v}^m_1\|_T \leq \beta_3(T) \epsilon^{1-q}.\]

We have
\[dr^m_2(t) = \big( L_t r^m_2(t) + r^m_{2,1}(t) + r^m_{2,2}(t) + r^m_{2,3}(t)\big) dt,\]
where
\begin{align*}
\epsilon^2 r^m_{2,1}(t)
& = w\ast\Big(F(\twp(\cdot-\varphi^m_1(t))+\tilde{v}^m_1(t))-F(\twp(\cdot-\varphi^m_1(t)))\\
& \qquad -F'(\twp(\cdot-\varphi^m_1(t)))\tilde{v}^m_1(t) - \frac{1}{2}F''(\twp(\cdot-\varphi^m_1(t)))(\tilde{v}^m_1(t))^2 \Big) \\
& \quad + w\ast\Big( \big(F'(\twp(\cdot-\varphi^m_1(t)))-F'(\twp(\cdot-\varphi^m_0(t)))\big) \tilde{v}^m_1(t)\Big)\\
& \quad + w\ast\Big( \big( F'(\twp(\cdot-\varphi^m_0(t)))-F'(\twp(\cdot-ct)) \\
& \qquad + \epsilon C^m_0(t) F''(\twp(\cdot-ct))\twp_x(\cdot-ct)\big)\tilde{v}^m_1(t)\Big) \\
& \quad - \epsilon C^m_0(t) w\ast\Big( F''(\twp(\cdot-ct))\twp_x(\cdot-ct) (\tilde{v}^m_1(t)-\epsilon v^m_0(t))\Big)\\
& \quad + \frac{1}{2} w\ast\Big(\big( F''(\twp(\cdot-\varphi^m_1(t)))-F''(\twp(\cdot-ct))\big) (\tilde{v}^m_1(t))^2\Big) \\
& \quad + \frac{1}{2} w\ast\Big(F''(\twp(\cdot-ct)) ((\tilde{v}^m_1(t))^2 - \epsilon^2 (v^m_0(t))^2)\Big)\\
& = \sum_{i=1}^{6} r^m_{2,1,i}(t),\\
\epsilon^2 r^m_{2,2}(t)
& = (\epsilon c^m_0(t)+\epsilon^2 c^m_1(t)) (\twp_x(\cdot-\varphi^m_1(t))-\twp_x(\cdot-ct)) \\
& \qquad + \epsilon^2 c^m_0(t) C^m_0(t) \twp_{xx}(\cdot-ct).
\end{align*}
Now
\begin{align*}
 \epsilon^4 \|r^m_{2,1,1}(t)\|_{1+\rho_t}^2
 & \leq  \frac{1}{36} \|F^{(3)}\|_{\infty}^2(1+ K_{\rho}) \|w\|_{\infty} \|\tilde{v}^m_1(t)\|_{1+\rho_t}^2 \int (\tilde{v}^m_1)^4(x,t) dx \\
 & \leq \frac{1}{36} \|F^{(3)}\|_{\infty}^2 (1+ K_{\rho}) \|w\|_{\infty} \|\tilde{v}^m_1(t)\|_{1+\rho_t}^2 \|\tilde{v}^m_1(t)\|_{\infty}^2 \|\tilde{v}^m_1(t)\|^2 \\
 & \leq \frac{1}{36} \|F^{(3)}\|_{\infty}^2 (1+ K_{\rho}) \|w\|_{\infty} K^2 \|\tilde{v}^m_1\|_T^6.
\end{align*}
Concerning $r^m_{2,1,4}$ and $r^m_{2,1,6}$, note that there exists $\tilde{\xi}(x,t)$ such that
\begin{align*}
\tilde{v}^m_1(t) 
& = \tilde{v}^m_0(t)  + \twp(\cdot-\varphi^m_0(t))-\twp(\cdot-\varphi^m_1(t)) \\
& = \epsilon v^m_0(t) + \epsilon r^m_1(t)+\epsilon^2 C^m_1(t) \twp_x(\cdot-\varphi^m_0(t)+\tilde{\xi}),
\end{align*}
and hence, using (\ref{eq:Cm1estimate}), for some constant $\beta_4(T)$.
\[\|\tilde{v}^m_1-\epsilon v^m_0\|_T \leq \beta_4(T) \epsilon^{2-2q}.\]
We have 
\[\epsilon^2 r^m_{2,2}(t) = R^m_{2,2}(t) + \int_0^t P_{t,s}(L_s+c\partial_x)R^m_{2,2}(s)ds\]
with
\begin{align*}
R^m_{2,2}(t) 
& = -\twp(\cdot-\varphi^m_1(t))+\twp(\cdot-ct) - (\epsilon C^m_0(t)+\epsilon^2 C^m_1(t)) \twp_x(\cdot-ct)\\
& \qquad + \epsilon^2 \frac{1}{2}(C^m_0)^2(t)\twp_{xx}(\cdot-ct).
\end{align*}
Now all the terms can be estimated as in the proof of Thm. \ref{thm:expansion1} and
we obtain that there exists 
$\alpha_2(T)$ independent of $m$ and $\epsilon$ such that
\[\|r^m_2\|_T \leq \alpha_2(T) \epsilon^{1-3q}. \]
\end{proof}

\begin{remark}
Note that in the case $k=1$ (Thm. \ref{thm:expansion1}) we could also define the stopping time $\tau_{q,\epsilon}$ as an exit time of the $L^2(1+\rho_t)$-norm of $v$ and obtain that 
$ \sup_{t\leq T} \|r^m_1(t)\|_{1+\rho_t} \leq \alpha_1(T) \epsilon^{1-2q}$. This is not the case in the proof of Thm. \ref{thm:expansion2}, where we need to control $\| \tilde{v}^m_1(t)\|_{\infty}$.
\end{remark}

%% file: measurerho.tex
\subsection{The Measure $\rho$}
\label{subsection:measurerho}

In order to be able to control the $L^2$-norm as well as the $L^2(\rho)$-norm of $v$, we will from now on work in the space $H^1(1+\rho) = H^1 \cap H^1(\rho)$ equipped with the norm $\|h\|_{H^1(1+\rho)} = \Big( \|h\|_{H^1}^2 + \|h\|_{H^1(\rho)}^2\Big)^{\frac{1}{2}}$.
To obtain a solution to (\ref{eq:snfedifference}) in $H^1(1+\rho)$ we adapt our assumptions on the noise.
We assume that $Q^{\frac{1}{2}}$ is Hilbert-Schmidt as an operator from $L^2$ into $H^1(1+\rho)$. $Q^{\frac{1}{2}}$ may for example be given by a symmetric integral kernel, $Q^{\frac{1}{2}}u = \int q(x,y) u(y) dy$, with $\int \norm{q(x,\cdot)}_{H^1(1+\rho)}^2 dx < \infty$. 
Then for any orthonormal basis $(e_k)$ of $L^2$, using Parseval's identity,
\begin{align*}
 \sum_k \|Q^{\frac{1}{2}}e_k\|_{H^1(1+\rho)}^2 
& = \sum_k \int \Big( \int q(x,y)e_k(y)dy \Big)^2 \\
& \qquad  + \Big(\int q_x(x,y) e_k(y) dy\Big)^2 (1+\rho(x)) dx \\
& = \int \big( \|q(x,\cdot)\|^2 + \|q_x(x,\cdot)\|^2 \big) (1+\rho(x)) dx \\
& = \int \|q(\cdot,y)\|_{H^1(1+\rho)}^2 dy < \infty
\end{align*}

We make the following assumptions on $\rho$.
\begin{assumption}
\label{ass:rho}
\begin{enumerate}[(i)]
	\item There exists a constant $L_{\rho}$ such that for all $x\in \R$ and $y>0$,
	\begin{equation}
	\label{eq:rho1}
	\rho(x-y) \leq L_{\rho} \rho(x).
	\end{equation}
	\item There exists a constant $K_{\rho}>0$ such that
		\begin{equation}
		\label{eq:rho2}
			w\ast\rho(x) \leq K_{\rho} \rho(x).
		\end{equation}
\end{enumerate}
\end{assumption}

Condition (i) says that $\rho$ should be roughly increasing. We have an a priori bound on the growth of $\rho$ since
$L^{\#,*}\psi =0$ implies that
\[\psi(x) = \int_0^{\infty} e^{-s} F'(\twp(x-cs)) w\ast\psi(x-cs) ds\]
and thus
\[|\psi_x(x)| \leq \Big( \Big\| \frac{F''(\twp)\twp_x}{F'(\twp)}\Big\| + \Big\| \frac{w_x}{w}\Big\| \Big) \psi(x),\]
and similarly \[|\twp_{xx}(x)| \leq \Big\|\frac{w_x}{w}\Big\| \twp_x(x), \]
such that
\begin{equation}
\label{eq:boundrhox}
|\rho_x| = \Big| \frac{\psi_x}{\psi} - \frac{\twp_{xx}}{\twp_x}\Big| \rho \leq \Big(\Big( \Big\| \frac{F''(\twp)\twp_x}{F'(\twp)}\Big\| + 2 \Big\| \frac{w_x}{w}\Big\| \Big) \rho =:M\rho.
\end{equation}

Condition (ii) says roughly that $\rho$ should neither grow nor decay too quickly relative to $w$. 	It has already been noted in \cite{faugerasinglis} that this is a sufficient condition for the existence of an $L^2(\rho)$-	valued solution.

In section \ref{section:asymptoticbehavior} we prove that the conditions are satisfied in the case of exponential synaptic decay.

\begin{proposition}
\label{prop:exunirho}
Assume that $\rho$ satisfies (\ref{eq:rho2}). Then
for any $\eta \in H^1(1+\rho)$ there exists a unique strong $H^1(1+\rho)$-valued solution $v$
to the stochastic evolution equation (\ref{eq:snfedifference}).
$v$ admits a continuous modification and for all $p\geq 1$, 
\[E(\sup_{0\leq t\leq T} \|v(t)\|_{H^1(1+\rho)}^p)<\infty.\]
\end{proposition}
\begin{proof} As in the proof of Proposition \ref{prop:exuni}, it is enough to show that $B(t,v) := -v + w\ast(F(\twp(\cdot-ct)+v)-F(\twp(\cdot-ct)))$ is Lipschitz continuous in $v$. This follows from the fact that for $v_1, v_2 \in H^1$, , using (\ref{eq:rho2}),
\begin{align*}
	& \Big\| w\ast \Big(F(\twp(y-ct)+v_1(y)) - F(\twp(y-ct)+v_2(y))\Big) \Big\|_{H^1(1+\rho)}^2\\
	& \leq \Big(1+ \normsup{\frac{w_x}{w}}^2\Big) \norm{F'}_{\infty}^2 \int \int w(x-y) (1+\rho(x)) dx (v_1(y)-v_2(y))^2 dy \\
	& \leq \Big(1+ \normsup{\frac{w_x}{w}}^2\Big) \norm{F'}_{\infty}^2 \int (1 + K_{\rho}\rho(y))) (v_1(y)-v_2(y))^2 dy \\
	& \leq \Big(1+ \normsup{\frac{w_x}{w}}^2\Big) \norm{F'}_{\infty}^2 (1+ K_{\rho}) \|v_1-v_2\|_{1+\rho}^2.
\end{align*}
\end{proof}

The family of linear operators defined in (\ref{eq:Lt}) satisfies 
\[\|L_th\|^2_{H^1(1+\rho)} \leq 2 \Big( 1+\Big( 1+\Big|\frac{w_x}{w}\Big\|_{\infty}^2 \Big) (1+K_{\rho}) \|F'\|_{\infty}^2 \Big) \|h\|_{H^1(1+\rho)}^2 =: L_*^2 \|h\|_{H^1(1+\rho)}^2.\]
It generates an evolution semigroup $(P_{t,s})_{0\leq s\leq t\leq T}$ with 
\[\|P_{t,s}h\|_{H^1(1+\rho)} \leq e^{L_*(t-s)} \|h\|_{H^1(1+\rho)}.\]
$v$ can thus be represented as a mild solution
\[v(t) = \epsilon P_{t,0} \eta + \int_0^t P_{t,s} R(s, v(s)) ds + \epsilon \int_0^t P_{t,s} dW_s.\]

%% file: immediaterelaxation.tex
\subsection{Immediate Relaxation}
\label{subsection:immediaterelaxation}

We now go over to the limit $m\rightarrow \infty$, presumably corresponding to immediate relaxation to the right phase. Since all the estimates in section \ref{subsection:expansion} are independent of $m$, the expansion will translate to the limiting regime once we have computed the limits of the coefficients.

Denote by $\pi_s$ the projection onto the orthogonal complement of $\twp_x(\cdot-cs)$ in $L^2(\rho_s)$, i.e.,
\[\pi_s h = h - \langle h, \twp_x \rangle_{\rho(\cdot-cs)} \twp_x(\cdot-cs).\]
Note that while $C^m_0(t) = \int_0^t c^m_0(s) ds$ is a process of bounded variation, in the limit $m\rightarrow \infty$ we go over to a process of unbounded variation.
The convergence is only locally uniform on $(0,T)$ due to the initial jump to the right phase in the limit.

\begin{lemma}
\label{lemma:convminfty}
For any $\delta>0$, for $i=1,2$, almost surely
\[\sup_{\delta\leq t \leq T}|C^m_i(t)-C_i(t)|\xrightarrow{m\rightarrow\infty}0\]
and
\[\sup_{\delta\leq t \leq T} \norm{v^m_i(t)-v_i(t)}_{H^1(1+\rho_t)} \xrightarrow{m\rightarrow\infty} 0,\]
where $C_0(0)=0$, $v_0(0)=\eta$,  and for $t>0$,
\begin{align*}
C_0(t) & = - \langle \eta, \psi \rangle - \int_0^t \langle \psi(\cdot-cs), dW_s \rangle \\
v_0(t) & = P_{t,0} \pi_0 \eta + \int_0^t P_{t,s} \pi_s dW_s,
\end{align*}
and where 
\begin{align*}
C_1(t) 
& = -  \int_0^t \langle w\ast\Big(F''(\twp(\cdot-cs))\big(\frac{1}{2} v_0^2(s) - C_0(s) \twp_x(\cdot-cs) v_0(s) \big)\Big), \psi(\cdot-cs)\rangle ds \\
& \quad  + C_0(t) \langle v_0(t), \psi_x(\cdot-ct)\rangle -\frac{1}{2} C_0^2(t) \langle \twp_x, \psi_x \rangle,\\
v_1(t)
& = \int_0^t P_{t,s} w\ast\Big( F''(\twp(\cdot-cs)) \big(\frac{1}{2} v_0^2(s) - C_0(s) \twp_x(\cdot-cs) v_0(s)\big)\Big) ds\\
& \quad  + C_1(t) \twp_x(\cdot-ct) - \frac{1}{2}\int_0^t P_{t,s}\twp_{xx}(\cdot-cs) dC_0^2(s)
\end{align*}
\end{lemma}
We postpone the proof to the end of this section.

\begin{theorem}
\label{thm:immediaterelax}
Let $\tau_{q,\epsilon}$ be as in (\ref{eq:tau}) and set $\tau_{q,\epsilon}^{\infty} = \inf \{0\leq t\leq T: |C_0(t)|\geq \epsilon^{-q}\}$. Then on $\{\tau_{q,\epsilon} \wedge \tau_{q,\epsilon}^{\infty}=T\}$,
\[u(x,t) = \twp(x-ct-\epsilon C_0(t)) + \epsilon v_0(x,t) + \epsilon r_1(x,t),\]
and if $F\in \mathcal{C}^3$, then 
\[u(x,t) = \twp(x-ct-\epsilon C_0(t)-\epsilon^2 C_1(t)) + \epsilon v_0(x,t) + \epsilon^2 v_1(x,t) + \epsilon^2 r_2(x,t),\]
where for $k=1,2$,
\[\norm{r_k}_T \leq \alpha_k(T) \epsilon^{1-(k+1)q},\]
with $\alpha_k$ as in Thms. \ref{thm:expansion1} and \ref{thm:expansion2}, and \[P(\tau_{q,\epsilon} \wedge \tau_{q,\epsilon}^{\infty}=T) \xrightarrow{\epsilon\rightarrow 0} 1.\]
\end{theorem}

\begin{proof}
Let $0<t < \tau_{q,\epsilon} \wedge \tau_{q,\epsilon}^{\infty}$.
Integrating (\ref{eq:cm0}) we obtain that
\begin{equation}
\label{eq:Cm0}
C^m_0(t) = -(1-e^{-mt} ) \langle \eta,\psi\rangle - \int_0^t (1-e^{-m(t-s)}) \langle \psi(\cdot-cs),dW_s\rangle .
\end{equation}
By It\^{o}'s Lemma,
\begin{align*}
C^m_0(t) & = -(1-e^{-mt} ) \langle \eta,\psi\rangle  -\int_0^t me^{-m(t-s)}  \langle \psi(\cdot-cs),W_s\rangle ds\\
& \qquad - c \int_0^t  (1-e^{-m(t-s)}) \langle \psi_x(\cdot-cs),W_s\rangle ds.
\end{align*}
Therefore, for $0<\delta<t$,
\[|C^m_0|_{\delta} \leq |\langle \eta,\psi\rangle| + \|\psi\| \|W\|_{\delta} + c \delta \|\psi_x\|\|W\|_{\delta} \xrightarrow{\delta\rightarrow 0} |\langle \eta,\psi\rangle| < \epsilon^{-q}. \]
Since for any $\delta>0$
$\sup_{\delta \leq s\leq t} |C^m_0(s)| \xrightarrow{m\rightarrow\infty} \sup_{\delta\leq s \leq t} |C_0(s)| < \epsilon^{-q} $ it follows that also $ t < \tau_{q,\epsilon} \wedge \tau_{q,\epsilon}^m $ for sufficiently large $m$.
Therefore, using Theorem \ref{thm:expansion1} and Lemma \ref{lemma:convminfty},
\begin{align*}
	 \norm{\epsilon r_1(t)}_{H^1(1+\rho_t)} 
	& \leq  \norm{u(t)-\twp(\cdot-ct-\epsilon C^m_0(t))-\epsilon v^m_0(t)}_{H^1(1+\rho_t)} \\
	& \quad {} + \norm{\twp(\cdot-ct-\epsilon C^m_0(t))-\twp(\cdot-ct-\epsilon C_0(t))}_{H^1(1+\rho_t)}\\
	& \quad+ \epsilon \norm{v^m_0(t)- v_0(t)}_{H^1(1+\rho_t)}\\
	& \leq \alpha_1(T) \epsilon^{2-2q} + \epsilon |C^m_0(t)-C_0(t)| (L_{\rho} \vee e^{2M\epsilon^{1-q}})^{\frac{1}{2}} \normH{\twp_x} \\
	& \quad {} + \epsilon \norm{v^m_0(t)- v_0(t)}_{H^1(1+\rho_t)}\\
	& \xrightarrow{m\rightarrow\infty} \alpha_1(T) \epsilon^{2-2q} , a.s.
\end{align*}
Thus, on $\{\tau_{q,\epsilon} \wedge \tau_{q,\epsilon}^{\infty}=T\}$, $\|r_1\|_T \leq \alpha_1(T) \epsilon^{2-2q}$.

The proof for the higher order expansion is analogous. 
\[P(\tau_{q,\epsilon} \wedge \tau_{q,\epsilon}^{\infty}=T) \xrightarrow{\epsilon\rightarrow 0} 1\]
is proven as in Theorem \ref{thm:expansion1}.
\qquad \end{proof}

\input{stability}

\input{proofconvergence}

%% file: stability.tex
The term $- \langle \eta, \psi \rangle$ in $C_0$ accounts for the initial phase difference between $u(0)$ and $\twp$.
We have
\[\ Var(C_0(t)) = \int_0^t \langle \psi(\cdot-cs), Q \psi(\cdot-cs) \rangle ds \approx \langle \psi, Q\psi\rangle t\]
if the correlations are roughly translation invariant (they cannot be translation invariant since $Q$ is of finite trace).
This is in accordance with the analysis of Bressloff and Webber in \cite{bressloffwebber}.

Note that for $t>0$,
\begin{equation}
\label{eq:v0orthogonal}
\begin{split}
\langle v_0(t), \twp_x(\cdot-ct)\rangle_{\rho_t} 
& = \langle P_{t,0} \pi_0 \eta , \psi(\cdot-ct) \rangle + \langle \int_0^t P_{t,s} \pi_s dW_s, \psi(\cdot-ct)\rangle \\
& = \langle \pi_0\eta, P_{t,0}^*(\psi(\cdot-ct))\rangle + \int_0^t \langle P_{t,s}^* \psi(\cdot-ct), \pi_s dW_s\rangle\\
& = \langle \pi_0\eta, \psi\rangle + \int_0^t \langle \psi(\cdot-cs), \pi_s dW_s\rangle =0.
\end{split}
\end{equation}
In the frozen wave setting, $v_0^{\#}$ is thus orthogonal to $\twp_x$ in $L^2(\rho)$. Recall that the frozen wave operator $L^{\#}$ generates a contraction semigroup on $\twp_x^{\perp}$. For $t>0$ we can therefore write
\[v_0^{\#}(t) = P^{\#}_t \pi_0 \eta + \int_0^t P^{\#}_{t-s}\Phi_s \pi_s dW_s = P^{\#}_t \pi_0 \eta + \int_0^t P^{\#}_{t-s} \pi_0 \Phi_s dW_s.\]
Using (\ref{eq:contractionsg}) it follows that
\[\|v_0(t)\|_{\rho_t} = \|v_0^{\#}(t)\|_{\rho} \leq e^{-\kappa t} \|\eta\|_{\rho}  + \bigg\| \int_0^t P^{\#}_{t-s} \pi_0 \Phi_s dW_s\bigg\|_{\rho} \]
and hence
\[	E(\|v_0(t)\|_{\rho_t}^2) \leq 2 e^{-2\kappa t} \norm{\eta}_{\rho}^2  +2 \int_0^t \|P^{\#}_{t-s} \pi_0 \Phi_s Q^{\frac{1}{2}}\|_{L_2(L^2,L^2(\rho))}^2 ds.\]
Let $(e_k)$ be an orthonormal basis of $L^2$. We have
\begin{align*}
\|P^{\#}_{t-s} \pi_0 \Phi_s Q^{\frac{1}{2}}\|_{L_2(L^2,L^2(\rho))}^2
& = \sum_k \|P^{\#}_{t-s} \pi_0 \Phi_s Q^{\frac{1}{2}}e_k\|_{\rho}^2 \leq e^{-\kappa(t-s)} \sum_k \| Q^{\frac{1}{2}}e_k\|_{\rho_s}^2 \\
& \leq L_{\rho} e^{-2\kappa(t-s)} \sum_k \| Q^{\frac{1}{2}}e_k\|_{\rho}^2 \leq L_{\rho} e^{-\kappa(t-s)} tr(Q) 
\end{align*}
and thus
\[	E(\|v_0^{\#}(t)\|_{\rho}^2) \leq 2 e^{-2\kappa t} \norm{\eta}_{\rho}^2  + 2L_{\rho} \frac{1}{2\kappa} (1-e^{-2\kappa t}) tr(Q) \xrightarrow{t\rightarrow \infty} \frac{L_{\rho} tr(Q)}{\kappa}.\]
$v^{\#}_0	$ is thus a stationary Ornstein-Uhlenbeck process on $\twp_x^{\perp}$.

So far it is not clear that the expansion in Theorem \ref{thm:immediaterelax} gives the right description of the influence of the noise on the traveling wave. A different choice of $C_0$ would yield another expansion and we are left to justify that our particular choice of $C_0$ provides the right picture.

Set $\varphi_k(t) = ct+\epsilon C_0(t) + \ldots + \epsilon^k C_{k}(t)$.
The $C_k$ describe the phase shift caused by the noise to order $\epsilon^{k+1}$ in the sense of the following proposition.

\begin{proposition}
\label{prop:optimalphase} For $t< \tau_{q,\epsilon}\wedge \tau_{q,\epsilon}^{\infty}$,
$a \mapsto \|u-\twp(\cdot-ct-\epsilon a)\|_{\rho_t}$ is locally minimal to order $\epsilon$ at $a=C_0(t)$.

$a \mapsto \|u-\twp(\cdot-ct-\epsilon C_0(t) -\epsilon^2 a)\|_{\rho_t(\cdot-\epsilon C_0(t))}$ is locally minimal to order $\epsilon^2$ at $a=C_1(t)$.
\end{proposition}

\begin{proof}
Note that, using (\ref{eq:v0orthogonal}) and Thm. \ref{thm:expansion1},
\begin{align*}
 \frac{1}{2}\frac{d}{da}\bigg|_{a=C_0(t)} \|u(t)-\twp(\cdot-ct-\epsilon a)\|_{\rho_t}^2
& = \epsilon^2 \langle v_0(t)+r_1(t), \twp_x(\cdot-\varphi_0(t))\rangle_{\rho_t} \\
& = \epsilon^2 \langle v_0(t), \psi(\cdot-ct) \rangle + o(\epsilon^2) = o(\epsilon^2)
\end{align*}
and
\[\frac{1}{2}\frac{d^2}{da^2}\bigg|_{a=C_0(t)} \|u(t)-\twp(\cdot-ct-\epsilon a)\|_{\rho_t}^2 = \epsilon^2 \langle \twp_x, \psi\rangle + o(\epsilon^2) = \epsilon^2 + o(\epsilon^2).\]
$C_0$ is thus such that $\|u(t)-\twp(\cdot-ct-\epsilon C_0(t))\|_{\rho_t}$ is locally minimal to order $\epsilon$.
Similarly we have that, using Thm. \ref{thm:expansion2},
\begin{align*}
& \frac{1}{2}\frac{d}{da}\bigg|_{a=C_1(t)} \|u(t)-\twp(\cdot-ct-\epsilon C_0(t)-\epsilon^2 a)\|_{\rho_t(\cdot-\epsilon C_0(t))}^2 \\
& = \epsilon^3 \langle v_0(t)+ \epsilon v_1(t)+ \epsilon r_2(t), \twp_x(\cdot-\varphi_1(t))\rangle_{\rho_t(\cdot-\epsilon C_0(t))} \\
& =  \epsilon^4  \big( - C_0(t) \langle v_0(t), \psi_x(\cdot-ct)\rangle + \langle v_1(t), \psi(\cdot-ct)\rangle\big) +  o(\epsilon^4).
\end{align*}
Note that 
\begin{align*}
& \langle v_1(t), \psi(\cdot-ct)\rangle \\
& =\int_0^t \langle w\ast\Big(F''(\twp(\cdot-cs))\big(\frac{1}{2}v_0^2(s)- C_0(s)\twp_x(\cdot-cs)v_0(s)\big)\Big), \psi(\cdot-cs)\rangle ds \\
& \qquad {} + C_1(t)  + \frac{1}{2}\langle \twp_x, \psi_x \rangle C_0^2(t)\\
& = C_0(t) \langle v_0(t), \psi_x(\cdot-ct).
\end{align*}
We thus obtain
\[\frac{1}{2}\frac{d}{da}\bigg|_{a=C_1(t)} \|u(t)-\twp(\cdot-ct-\epsilon C_0(t)-\epsilon^2 a)\|_{\rho_t(\cdot-\epsilon C_0(t))}^2 = o(\epsilon^4)\]
and
\[\frac{1}{2}\frac{d^2}{da^2}\bigg|_{a=C_1(t)} \|u-\twp(\cdot-ct-\epsilon C_0(t)-\epsilon^2 a\|_{\rho_t(\cdot-\epsilon C_0(t))}^2 = \epsilon^4 + o(\epsilon^4). \]
\end{proof}

Together with Proposition \ref{prop:optimalphase} and the properties of $v_0^{\#}$, the expansion expresses the stability of the traveling wave under the noise. With large probability, up to the time horizon $T$, the stochastic solution can be described as a wave profile moving at an adapted speed with stationary fluctuations around it.

\begin{remark}
The spectral gap of $L^{\#} $ expresses linear stability properties. The control over $\|v_0(t)\|_{\rho_t}$ allows us to derive local stability up to the time horizon $T$, since the rest terms are of smaller order. The main problem in going over to larger time scales is that we lose control of the $L^2$-norm in estimates such as
\[ \|w\ast v^2\|_{\rho}^2 \leq K_{\rho} \|v\|_{\rho}\|v\|.\] 
\end{remark}

%% file: proofconvergence.tex
{\em Proof of Lemma \ref{lemma:convminfty}}.
Using (\ref{eq:Cm0}) we obtain that for $t>0$,
\[C^m_0(t)-C_0(t) = e^{-mt} \langle \eta, \psi \rangle + \int_0^t e^{-m(t-s)} \langle \psi(\cdot-cs), dW_s\rangle =:e^{-mt} \langle \eta, \psi \rangle + S_t.\]
By It\^{o}'s Lemma,
\begin{align*}
	S_t 
	&= \langle \psi(\cdot-ct), W_t \rangle - \int_0^t m e^{-m(t-s)} \langle \psi(\cdot-cs), W_s\rangle ds \\
	& \qquad + \int_0^t c e^{-m(t-s)} \langle \psi_x(\cdot-cs), W_s\rangle ds \\
	&= \int_0^t me^{-m(t-s)} \big( \langle \psi(\cdot-ct), W_t \rangle - \langle \psi(\cdot-cs), W_s\rangle \big) ds + e^{-mt} \langle \psi(\cdot-ct), W_t\rangle \\
	& \qquad {} + \int_0^t c e^{-m(t-s)} \langle \psi_x(\cdot-cs), W_s\rangle ds.
\end{align*}
Note that by the H\"{o}lder continuity of $t\rightarrow \langle \psi(\cdot-ct), W_t\rangle$, for any $0 <\beta < \frac{1}{2}$,  $M_{\beta}(T,\omega):= \sup_{|t-s|\leq T} \frac{|\langle \psi(\cdot-ct), W_t \rangle - \langle \psi(\cdot-cs), W_s\rangle|}{|t-s|^{\beta}} < \infty$ almost surely (cf. \cite{daprato}, Thm. 3.3).
We can thus estimate
\begin{align*}
	& \Big| \int_0^t m e^{-m(t-s)} \big( \langle \psi(\cdot-ct), W_t \rangle - \langle \psi(\cdot-cs), W_s\rangle \big) ds \Big| \\
	& \leq M_{\beta}(T,\omega) \int_0^t m e^{-m(t-s)} (t-s)^{\beta} ds \\
	& \leq M_{\beta}(T,\omega) \frac{1}{m^{\beta}} \int_0^{\infty} e^{-r} r^{\beta} dr = M_{\beta}(T,\omega) \frac{1}{m^{\beta}} \Gamma(1+\beta),
\end{align*}
where $\Gamma(t) =  \int_0^{\infty} x^{t-1} e^{-x} dx$ is the gamma function,
and we obtain that
\begin{align*}
	 |S_t| 
	& \leq M_{\beta}(T,\omega) \frac{1}{m^{\beta}} \Gamma(1+\beta) + (e^{-mt} \norm{\psi} + \frac{c}{m} \norm{\psi_x}) \sup_{0\leq s\leq t}\|W(s)\| .
\end{align*}
Thus,
\begin{align*}
& \sup_{\delta\leq t \leq T}|C^m_0(t)-C_0(t)|\\
& \leq e^{-m\delta} \|\psi\| \big(\|\eta\|  + \sup_{0 \leq t\leq T }\|W(t)\| \big) \\
& \qquad {} + \frac{c}{m} \norm{\psi_x} \sup_{0\leq t\leq T}\|W(t)\| + \frac{1}{m^{\beta}} M_{\beta}(T,\omega)  \Gamma(1+\beta) \xrightarrow{m\rightarrow\infty}0, a.s.
\end{align*}

Now we consider $v^m_0$. Since for $s\leq t$, $P_{t,s}\twp_x(\cdot-cs) = \twp_x(\cdot-ct)$, we have for $t>0$
\begin{align*}
v^m_0(t)
& = P_{t,0} \eta + \int_0^t c^m_0(s) P_{t,s} \twp_x(\cdot-cs) ds + \int_0^t P_{t,s} dW_s \\
& = P_{t,0} \pi_0 \eta + \langle \eta, \psi\rangle \twp_x(\cdot-ct) + C^m_0(t) \twp_x(\cdot-ct)\\
& \quad  + \int_0^t P_{t,s} \pi_s dW_s + \int_0^t \langle \psi(\cdot-cs), dW_s\rangle \twp_x(\cdot-ct) \\
& = v_0(t) + (C^m_0(t)-C_0(t))\twp_x(\cdot-ct),
\end{align*}
and hence
\begin{align*}
	& \sup_{\delta\leq t \leq T} \norm{v^m_0(t)-v_0(t)}_{H^1(1+\rho_t)} \\
	& \leq \sup_{\delta\leq t \leq T} |C^m_0(t)-C_0(t)| \norm{\twp_x}_{H^1(1+\rho)} \xrightarrow{m\rightarrow\infty} 0, a.s.
\end{align*}

The convergence of $C^m_1$ follows from Lemma \ref{lemma:C1m} and the convergence of $C^m_0$ and $v^m_0$.

Using (\ref{eq:sdev1m}) and
\begin{align*}
&   - c^m_0(s) C^m_0(s) P_{t,s} \twp_{xx}(\cdot-cs) + c^m_1(s) P_{t,s}\twp_x(\cdot-cs) \\
& = \big(\frac{d}{ds} P_{t,s} + P_{t,s}(L_s+c\partial_x) \big) \big( - \frac{1}{2}(C^m_0)^2(s) \twp_{xx}(\cdot-cs) + C^m_1(s) \twp_x(\cdot-cs)\big)
\end{align*}
we obtain
\begin{align*}
v^m_1(t)
& = \int_0^t P_{t,s} w\ast\Big(F''(\twp(\cdot-cs)) \big(\frac{1}{2}(v^m_0)^2(s)-C^m_0(s) \twp_x(\cdot-cs)v^m_0(s)\big)\Big) ds \\
& \quad - \frac{1}{2}(C^m_0)^2(t)\twp_{xx}(\cdot-ct) + C^m_1(t) \twp_x(\cdot-ct) \\
& \quad - \frac{1}{2} \int_0^t (C^m_0)^2(s)  P_{t,s} (L_s+c\partial_x) \twp_{xx}(\cdot-cs) ds.
\end{align*}
Using the convergence of $C^m_0, C^m_1$, and $v^m_0$, and the fact that by It\^{o}'s Lemma
\begin{align*}
& - \frac{1}{2}C_0^2(t)\twp_{xx}(\cdot-ct)- \frac{1}{2} \int_0^t C_0^2(s)  P_{t,s} (L_s+c\partial_x) \twp_{xx}(\cdot-cs) ds\\
& =  - \frac{1}{2}C_0^2(t)P_{t,t}\twp_{xx}(\cdot-ct) +\frac{1}{2} \int_0^t C_0^2(s)  \frac{d}{ds} (P_{t,s}  \twp_{xx}(\cdot-cs)) ds\\
& = - \frac{1}{2} \int_0^t P_{t,s}\twp_{xx}(\cdot-cs) dC_0^2(s), 
\end{align*}
the convergence of $v^m_1$ to $v_1$ follows.
\qquad\endproof

%% file: asymptoticbehavior.tex
\section{Asymptotic behavior of $\twp_x$, $\psi$, and $\rho$}
\label{section:asymptoticbehavior}

In all of this section we assume exponential synaptic decay, i.e. $w(x) = \frac{1}{2\sigma}e^{-\frac{|x|}{\sigma}}$ for some $\sigma > 0$. 
We will show that Assumption \ref{ass:rho}  on the density $\rho$ is satisfied for this choice of kernel.
To this end, we analyze the asymptotic behavior of $\twp_x$, $\psi$, and $\rho$.

We assume that there exist $z_1\leq z_2$ such that $F''(x) \geq 0$ for $x \leq z_1$, $F''(x) \leq 0$ for $x \geq z_2$.

Set $\phi(x) = w\ast\psi(x)$.
\begin{lemma}
\label{lemma:signpsix}
\begin{enumerate}[(i)]
\item There exist $y_1<y_2$ such that $\twp_{xx}(x)\geq0$ for $x < y_1$ and $\twp_{xx}(x)<0$ for $x > y_2$.
\item There exist $\tilde{y}_1<\tilde{y}_2$ such that $\phi_x(x) > 0$ for $ x<\tilde{y}_1$ and $\phi_x(x)<0$ for $x>\tilde{y}_2$.
\item For all $x\leq \bar{y}_1:=\min (\twp^{-1}(z_1), y_1, \tilde{y}_1)$, $\psi_x(x) \geq 0$, while for all $x \geq \bar{y}_2:=\max(\twp^{-1}(z_2), y_2, \tilde{y}_2)$, $\psi_x(x) \leq 0$.
\end{enumerate}
\end{lemma}
\begin{proof}
(i) 
We have
	\begin{align*}
		\twp(x) - \sigma^2 \twp_{xx}(x) & = (1-\sigma^2 \Delta ) \int_0^{\infty} e^{-s} w\ast F(\twp)(x+cs) ds \\
		& = \int_0^{\infty} e^{-s} F(\twp(x+cs)) ds \geq F(\twp(x)),
	\end{align*}
	which implies that $\sigma^2 \twp_{xx} \leq \twp - F(\twp) < 0$ for $x > \twp^{-1}(a)$.
	
	Let $b_1 = \min \left\{ x: F'(x) \geq 1 \right\}$.
	Assume that there exist $x_1 < x_2 < \twp^{-1}(b_1)$ such that
	$\twp_{xx}(x_1)=0, \twp_{xx}(x_2) = 0$ and $\twp_{xx}(x) < 0 $ for $x_1<x<x_2$.
	We have
	\begin{align*}
		 \twp-\sigma^2\twp_{xx}-c\twp_x+c\sigma^2\twp_{xxx}
		& = (1-\sigma^2 \Delta) (\twp - c\twp_x) \\
		&= (1-\sigma^2 \Delta) w\ast F(\twp) = F(\twp)
	\end{align*}
	and therefore
	\begin{align*}
		0 & = \sigma^2 (\twp_{xx}(x_2) - \twp_{xx}(x_1))  = \underbrace{c \sigma^2 (\twp_{xxx}(x_2)-\twp_{xxx}(x_1))}_{\geq 0} - \underbrace{ c (\twp_x(x_2) - \twp_x(x_1))}_{ < 0} \\
		&\qquad {}  + \underbrace{\twp(x_2) - F(\twp(x_2)) - (\twp(x_1)- F(\twp(x_1))) }_{= \int_{x_1}^{x_2} (1-F'(\twp(x)))\twp_x(x) dx > 0} > 0,
	\end{align*}
	which is a contradiction. 
Thus, since $\twp_x > 0$ implies that $\twp_{xx}(x) > 0$ for arbitrarily small $x$, the claim follows.

(ii): $\phi$ satisfies 
\begin{equation}
\label{eq:phi}
\begin{split}
F'(\twp)\phi  & = (1+c\partial_x) \psi = (1+c\partial_x) (1-\sigma^2\Delta) \phi \\
& = \phi + c\phi_x - \sigma^2 \phi_{xx} -c\sigma^2 \phi_{xxx}.
\end{split}
\end{equation}
There exist $z'_1 < z'_2$ such that $F'(\twp(x))<1$ for all $x\leq z'_1$ and $x\geq z'_2$.
Since $\int_{-\infty}^x\phi_x(y) dy = \phi(x)>0$ for all $x$, there exist arbitrarily small $z$ such that $\phi_x(z)>0$.
Analogously, since $\int_x^{\infty} \phi_x(y) dy = -\phi(x)<0$, there exist arbitrarily large $z$ such that $\phi_x(z)<0$.
We show that there exists no positive local maximum of $\phi_x$ on $\{{F'(\twp)<1}\}$. Then (ii) follows.

So assume there exists $x_0$ such that $F'(\twp(x_0)) < 1$ and $\phi_x$ attains a positive local maximum at $x_0$. Then, using (\ref{eq:phi}),
\[0 < (1-F'(\twp(x_0)))\phi(x_0) = \underbrace{-c\phi_x(x_0)}_{<0} + \underbrace{\sigma^2 \phi_{xx}(x_0)}_{=0} + c\sigma^2 \underbrace{\phi_{xxx}(x_0)}_{\leq 0} < 0,\]
which is a contradiction.

(iii): $\psi$ satisfies $\psi+c\psi_x = F'(\twp) \phi$.  Differentiating we obtain
\[\psi_x(x) + c\psi_{xx}(x) = F''(\twp(x)) \twp_x(x) \phi(x) + F'(\twp(x)) \phi_x(x) =:g(x).\]
For $x \leq \bar{y}_1$, $g(x) >0$ and for $x \geq \bar{y}_2$, $g(x) < 0$. Thus, $\psi$ does not attain a local maximum on $(-\infty, x_1)$, nor a local minimum on $(x_2,\infty)$.
Since $\psi>0$ there exist arbitrarily small $x$ such that $\psi_x(x)>0$ and arbitrarily large $x$ such that $\psi_x(x) <0$, and the claim follows.
\qquad\end{proof}

Let $\delta_1 = 1- \lim_{x\rightarrow -\infty} F'(\twp(x))$, $\delta_2 = 1-\lim_{x\rightarrow\infty} F'(\twp(x))$.

\begin{theorem}
\label{thm:asympbeh}
Let $\epsilon>0$. There exist $x_1(\epsilon) < x_2(\epsilon)$ and $\sqrt{\delta_1} <\tilde{\delta}_1(c) < 1, \sqrt{\delta}_2 < \tilde{\delta}_2(c) < 1$, such that for all $x\leq x_1, y >0$,
\begin{align*}
	{} & \twp_x(x) \leq e^{\frac{\sqrt{\delta_1}+\epsilon}{\sigma} y} \twp_x(x-y), \\
	e^{\frac{\tilde{\delta}_1(c)-\epsilon}{\sigma}y}\phi(x-y) \leq & \phi(x)  \leq e^{\frac{\tilde{\delta}_1(c)+\epsilon}{\sigma}y}\phi(x-y)
\end{align*}
and for all $x\geq x_2, y>0$,
\begin{align*}
	e^{-\frac{\tilde{\delta}_2(c)+\epsilon}{\sigma} y}\twp_x(x) \leq & \twp_x(x+y) \leq e^{-\frac{\tilde{\delta}_2(c) - \epsilon}{\sigma}y} \twp_x(x) \\
	e^{-\frac{\sqrt{\delta_2}+\epsilon}{\sigma}y} \phi(x) \leq & \phi(x+y) & {}.
\end{align*}
For $i=1,2$, $\tilde{\delta}_i(c)$ is the unique positive root of $f_i(x, c) = c x^3 + \sigma x^2 -cx - \delta_i c$, and
is increasing in $c$ with $\tilde{\delta}_i(0)=\sqrt{\delta_i}$ and $\lim_{c\rightarrow\infty} \tilde{\delta}_i(c) = 1$.
\end{theorem}
\begin{proof}
Let $\bar{y}_1,\bar{y}_2$ be as in Lemma \ref{lemma:signpsix}.
Note that
\begin{align*}
	\frac{\sigma^2\twp_{xxx}(x)}{\twp_x(x)} 
	& = 1-\frac{(I - \sigma^2 \Delta) \twp_x(x)}{\twp_x(x)} \\
	& = 1 - \frac{(I-\sigma^2 \Delta) \int_0^{\infty} e^{-s} w\ast(F'(\twp) \twp_x)(x+cs) ds}{\twp_x(x)} \\
	& = 1 - \frac{\int_0^{\infty} e^{-s} F'(\twp(x+cs))\twp_x(x+cs) ds}{\twp_x(x)},
\end{align*}
and since $F'(\twp)\phi = \psi + c\psi_x = (I+c\partial_x)(I-\sigma^2\Delta) \phi$,
\begin{align*}
	\frac{\sigma^2 \phi_{xx}(x)}{\phi(x)}
	& = 1 - \frac{(I+c\partial_x)^{-1}( F'(\twp)\phi)(x)}{ \phi(x) } \\
	& = 1 - \frac{\int_0^{\infty} e^{-s} F'(\twp(x-cs)) \phi(x-cs) ds}{\phi(x)}.
\end{align*}
So if $c=0$, then $\frac{\sigma^2 \twp_{xxx}(x)}{\twp_x(x)} = 1-F'(\twp(x))$, which converges to $\delta_1$ and $\delta_2$ for $x\rightarrow -\infty$ and $x\rightarrow\infty$, respectively.

Now assume that $c>0$. 
For $x\leq \bar{y}_1$, $\twp_{xx}(x) \geq 0$ and thus
\begin{align*}
	\frac{\sigma^2\twp_{xxx}(x)}{\twp_x(x)} 
	& \leq 1 - \int_0^{\frac{z_1-x}{c}} e^{-s} F'(\twp(x+cs))\frac{\twp_x(x+cs)}{\twp_x(x)} ds \\
	& \leq 1 - \int_0^{\frac{z_1-x}{c}} e^{-s} F'(\twp(x+cs)) ds \xrightarrow{x\rightarrow - \infty} \delta_1.
\end{align*}
Thus, there exists $x_1(\epsilon)$ such that for $x\leq x_1$,
\[\twp_{xxx}(x) \leq \frac{\delta_1+\epsilon}{\sigma^2}\twp_x(x),\]
and since
$\frac{d}{dx} (\twp_{xx}^2(x) - \frac{\delta_1+\epsilon}{\sigma^2}\twp^2_x(x)) = 2\twp_{xx}(x) (\twp_{xxx}(x) - \frac{\delta_1+\epsilon}{\sigma^2}\twp_x(x))\leq 0$
and $\lim_{x\rightarrow -\infty} (\twp_{xx}^2(x) - \frac{\delta_1+\epsilon}{\sigma^2}\twp^2_x(x)) = 0$, it follows that
$\twp_{xx}(x) \leq \frac{\sqrt{\delta_1+\epsilon}}{\sigma}\twp_x(x)$ and hence for $y>0$, $\twp_x(x) \leq e^{\frac{\sqrt{\delta_1+\epsilon}}{\sigma}y}\twp_x(x-y)$.

For $x \geq \bar{y}_2$, 
$\phi_x(x) \leq 0$ and thus
\begin{align*}
	\frac{\sigma^2\phi_{xx}(x)}{\phi(x)}
	& \leq 1 - \int_0^{\frac{x-z_2}{c}} e^{-s} F'(\twp(x-cs)) \frac{\phi(x-cs)}{\phi(x)} ds \\
	& \leq 1 - \int_0^{\frac{x-z_2}{c}} e^{-s} F'(\twp(x-cs)) ds \xrightarrow{x\rightarrow\infty} \delta_2
\end{align*}
and we obtain similarly to the above that there exists $x_2$ such that for $x\geq x_2$, $y>0$,
$\phi(x+y) \geq e^{-\frac{\sqrt{\delta_2+ \epsilon}}{\sigma}y}\phi(x)$.

Next we show that $\tilde{\delta}_1^2(c) := \lim_{x\rightarrow -\infty} \frac{\sigma^2 \phi_{xx}(x)}{\phi(x)}$ and $\tilde{\delta}_2^2(c) := \lim_{x\rightarrow \infty} \frac{\sigma^2 \twp_{xxx}(x)}{\twp_x(x)}$ exist.
Note that $|\phi_x| \leq \frac{1}{\sigma}\phi$ such that for $x \leq z_1$ and $y>0$, $\phi(x) \leq e^{\frac{1}{\sigma}y} \phi(x-y)$. It follows that for $x\leq z_1$,
\begin{align*}
	\frac{\sigma^2\phi_{xx}(x)}{\phi(x)}
	& \leq 1 - \int_0^{\infty} e^{-s} F'(\twp(x-cs)) e^{-\frac{1}{\sigma}cs} ds \\
	& \xrightarrow{x\rightarrow -\infty}  1 - (1-\delta_1) \frac{\sigma}{\sigma+c} = \frac{c}{\sigma+c} + \delta_1 \frac{\sigma}{\sigma+c}=:\delta_1^{(1)}(c),
\end{align*}
with $\delta_1 < \delta_1^{(1)}(c) <1$.
It follows that there exists $x_1$ such that for $x\leq x_1$, $y>0$,
$\phi(x) \leq e^{\frac{\sqrt{\delta_1^{(1)}(c) + \epsilon}}{\sigma}y} \phi(x-y)$.
Using this improved bound, we obtain that 
\begin{align*}
	\frac{\sigma^2\phi_{xx}(x)}{\phi(x)}
	& \leq 1 - \int_0^{\infty} e^{-s} F'(\twp(x-cs)) e^{-\frac{\sqrt{\delta_1^{(1)}(c) + \epsilon}}{\sigma}cs} ds \\
	& \xrightarrow{x\rightarrow -\infty}  1 - (1-\delta_1) \frac{\sigma}{\sigma+\sqrt{\delta_1^{(1)}(c)+\epsilon}c}=:\delta_1^{(2)}(c,\epsilon).
\end{align*}
Thus, $\limsup_{x\rightarrow - \infty}  \frac{\sigma^2\phi_{xx}(x)}{\phi(x)} \leq \delta_1^{(2)}(c,\epsilon) \xrightarrow{\epsilon\rightarrow 0} 1 - (1-\delta_1) \frac{\sigma}{\sigma+\sqrt{\delta_1^{(1)}(c)}c} =:\delta_1^{(2)}(c) $ with $\delta_1 < \delta_1^{(2)}(c) < \delta_1^{(1)}(c)$.
Iterating this procedure we obtain a decreasing sequence $\delta_1^{(n)}(c)> \delta_1$ satisfying
\[\delta_1^{(n+1)}(c) = 1 - (1-\delta_1) \frac{\sigma}{\sigma+\sqrt{\delta_1^{(n)}(c) }c}. \]
Thus, $\tilde{\delta}_1(c) := \lim_{n\rightarrow \infty} \sqrt{\delta_1^{(n)}(c)}$ satisfies
\[ c\tilde{\delta}_1^3(c) + \sigma \tilde{\delta}_1^2(c) - c\tilde{\delta}_1(c) = \delta_1 \sigma\]
and is therefore the unique positive root of $f_1(c,x) = c x^3 + \sigma x^2  -cx-\delta_1\sigma$.

On the other hand, for small enough $x$,
\[\frac{\sigma^2\phi_{xx}(x)}{\phi(x)} \geq 1 - \int_0^{\infty} e^{-s} F'(\twp(x-cs)) ds \xrightarrow{x\rightarrow -\infty} \delta_1,\]
and hence
\begin{align*}
	\frac{\sigma^2\phi_{xx}(x)}{\phi(x)} 
	& \geq 1 - \int_0^{\infty} e^{-s} F'(\twp(x-cs)) e^{-\frac{\sqrt{\delta_1-\epsilon}}{\sigma}cs}ds \\
	& \xrightarrow{x\rightarrow -\infty} 1 - (1-\delta_1) \frac{\sigma}{\sigma+ \sqrt{\delta_1-\epsilon}c}=: \delta_1'^{(1)}(c, \epsilon).
\end{align*}
Thus, $\liminf_{x\rightarrow - \infty} \frac{\sigma^2\phi_{xx}(x)}{\phi(x)} \geq \delta_1'^{(1)}(c, \epsilon) \xrightarrow{\epsilon\rightarrow 0} 1 - (1-\delta_1) \frac{\sigma}{\sigma+ \sqrt{\delta_1}c} =: \delta_1'^{(1)}(c) $ with $\delta_1 < \delta_1'^{(1)}(c) < 1$.
Iteration of the procedure yields an increasing sequence $\delta_1'^{(n)}(c)<1$ and $\delta_1'(c) := \lim_{n\rightarrow\infty} \sqrt{\delta_1'^{(n)}(c)}$ is the unique positive root of
$f_1(c,x) = c x^3 + \sigma x^2  -cx-\delta_1\sigma$. Hence, $\delta_1'(c) = \tilde{\delta}_1(c)$
and it follows that there exists $x_1$ such that for $x\leq x_1$, $y>0$,
\[e^{\frac{\tilde{\delta}_1(c)-\epsilon}{\sigma}y}\phi(x-y) \leq \phi(x)  \leq e^{\frac{\tilde{\delta}_1(c)+\epsilon}{\sigma}y}\phi(x-y)\]
Analogously, we obtain that there exists $x_2$ such that for $x \geq x_2$, $y>0$,
\[	e^{-\frac{\tilde{\delta}_2(x)+\epsilon}{\sigma} y}\twp_x(x) \leq \twp_x(x+y)  \leq e^{-\frac{\tilde{\delta}_2(c) - \epsilon}{\sigma}y} \twp_x(x), \]
where $\tilde{\delta}_2(c)$ is the unique positive root of $f_2(c,x) = c x^3 + \sigma x^2 -cx - \delta_2c$.

Since for $i=1,2$,
\[0 = \frac{d}{dc} f_i(c,\tilde{\delta}_i(c)) = \frac{\partial}{\partial c} f_i(c, \tilde{\delta}_i(c)) + \frac{d}{dc} \tilde{\delta}_i(c) \frac{\partial}{\partial x}f_i(c, \tilde{\delta}_i(c)),\]
$\frac{\partial}{\partial c} f_i(c, \tilde{\delta}_i(c)) = \tilde{\delta}_i^3(c)-\tilde{\delta}_i(c)< 0$, and $\frac{\partial}{\partial x}f(c, \tilde{\delta}_i(c)) > 0$, it follows that $ \tilde{\delta}_i(c)$ is increasing in $c$ with $\tilde{\delta}_i(0) = \sqrt{\delta_i}$ and $\lim_{c\rightarrow\infty} \tilde{\delta}_i(c) = 1$.
\qquad \end{proof}

\begin{proposition}
\label{prop:asympbehrho} 
\begin{enumerate}[(i)]
	\item There exists constants $\tilde{k}_1, \tilde{k}_2$ such that $\tilde{k}_1 \phi \leq \psi \leq \tilde{k}_2\phi$.
	\item Let $\epsilon>0$ (small enough) and let $x_1(\epsilon)<x_2(\epsilon)$ be as in Theorem \ref{thm:asympbeh}. There exist constants $k_1, k_2, k_1', k_2'$ such that for $x\leq x_1$, $y>0$, 
\[k_1 e^{\frac{\tilde{\delta}_1(c)- \sqrt{\delta_1}-2\epsilon}{\sigma}y} \rho(x-y) \leq \rho(x) \leq k_2 e^{\frac{\tilde{\delta}_1(c)+\epsilon}{\sigma}y}\rho(x-y)\]
and for $x\geq x_2$, $y>0$,
\[ k_1' e^{\frac{\tilde{\delta}_2(c)-\sqrt{\delta_2}-2\epsilon}{\sigma}y} \rho(x-y)  \leq \rho(x) \leq k_2' e^{\frac{\tilde{\delta}_2(c)+\epsilon}{\sigma}y}\rho(x-y).  \]
\end{enumerate}
\end{proposition}
\begin{proof}
(i) We have
\[(1+c\partial_x) \psi_x = F''(\twp) \twp_x w\ast\psi + F'(\twp)w_x\ast\psi\]
and thus
\begin{align*}
\psi_x(x)
& =  \int_0^{\infty} e^{-s}\big( F''(\twp(x-cs))\twp_x(x-cs) w\ast\psi(x-cs)\\
& \qquad +F'(\twp(x-cs))w_x\ast\psi(x-cs) \big) ds\\
& \leq \Big( \Big\| \frac{F''(\twp)\twp_x}{F'(\twp)}\Big\|_{\infty} + \Big\|\frac{w_x}{w}\Big\|_{\infty}\Big) \psi(x).
\end{align*}
It follows that there exists $\tilde{k}_1$ such that
\[\phi = \frac{\psi+c\psi_x}{F'(\twp)} \leq \frac{1}{\tilde{k}_1} \psi.\]

Let $\bar{y}_1, \bar{y}_2$ be as in Lemma \ref{lemma:signpsix}. Fix $\delta>0$.
Then for $x \geq \bar{y}_2+\delta$, 
\[\phi(x) = w\ast\psi(x) \geq \int_{x-\delta}^x w(x-y) \psi(y) \geq \int_0^{\delta} w(y)dy \psi(x),\]
and for $x\leq \bar{y}_1-\delta$,
\[\phi(x) \geq \int_x^{x+\delta} w(x-y)  \psi(y) dy = \int_0^{\delta} w(y)dy \psi(x). \]
For $x_1-\delta\leq x \leq x_2+\delta$,
\[\phi(x) \geq \frac{\min_{x_1-\delta\leq y\leq x_2+\delta} \phi(y)}{\max_{x_1-\delta\leq y \leq x_2+\delta} \psi(y)}\psi(x),\]
and the claim follows.

(ii)
For $x\leq x_1$, $y>0$,
\[\rho(x-y) \leq \tilde{k}_2 \frac{\phi(x-y)}{\twp_x(x-y)} \leq \tilde{k}_2 e^{-\frac{\tilde{\delta}_1(c) -\sqrt{\delta_1}-2\epsilon}{\sigma}y}\frac{\phi(x)}{\twp_x(x)} \leq \frac{\tilde{k}_2}{\tilde{k}_1}  e^{-\frac{\tilde{\delta}_1(c) -\sqrt{\delta_1}-2\epsilon}{\sigma}y} \rho(x)\]
and
\[\rho(x-y) \geq \tilde{k}_1 \frac{\phi(x-y)}{\twp_x(x-y)} \geq \tilde{k}_1 e^{- \frac{\tilde{\delta}_1(c) + \epsilon}{\sigma} y}\frac{\phi(x)}{\twp_x(x)} \geq \frac{\tilde{k}_1}{\tilde{k}_2} e^{- \frac{\tilde{\delta}_1(c) + \epsilon}{\sigma} y} \rho(x).\]
For $x\geq x_2$, $y>0$,
\[\rho(x+y) \leq \tilde{k}_2 e^{\frac{\tilde{\delta}_2(c)+\epsilon}{\sigma}y} \frac{\phi(x)}{\twp_x(x)} \leq  \frac{\tilde{k}_2}{\tilde{k}_1}e^{\frac{\tilde{\delta}_2(c)+\epsilon}{\sigma}y} \rho(x),  \]
and
\[\rho(x+y) \geq \frac{\tilde{k}_1}{\tilde{k}_2} e^{\frac{\tilde{\delta}_2(c)-\sqrt{\delta_2}-2\epsilon}{\sigma}y}  \rho(x). \]
\end{proof}

\begin{corollary}
\label{lemma:boundrhot}
There exists a constant $L_{\rho}$ such that
$ \rho(x-y) \leq L_{\rho} \rho(x)$ for all $x\in\R$ and $y\geq 0$.
\end{corollary}
\begin{proof}
By Proposition \ref{prop:asympbehrho}, there exist $x_1 < x_2$ and a constant $k$ such that
$\rho(x-y) \leq k \rho(x)$ for $x \leq x_1, y>0$, and $\rho(x-y) \leq k \rho(x)$ for $x-y\geq x_2, y >0$.
If $x-y \leq x_1 \leq x \leq x_2$, then
$\rho(x-y) \leq k \rho(x_1) \leq  k \frac{\rho(x_1)}{\min_{x_1\leq z \leq x_2} \rho(z)}\rho(x)$,
if $x-y\leq x_1 < x_2 \leq x$, then 
$\rho(x-y) \leq k \rho(x_1) \leq  k^2 \frac{\rho(x_1)}{\rho(x_2)}\rho(x)$,
and if $x_1 \leq x-y \leq x_2$, then
$\rho(x-y) \leq k \frac{\max_{x_1\leq z \leq x_2} \rho(z)}{\min_{x_1\leq z \leq x_2} \rho(z)}\rho(x)$.
\end{proof}

\begin{corollary}
\label{lemma:rho}
There exists a constant $K_{\rho}$ such that for all $x \in \R$,
\[\int w(x-y) \rho(y) dy \leq K_{\rho} \rho(x).\]
\end{corollary}
\begin{proof}
Fix $\epsilon>0$ (small enough) and let $x_1, x_2$ be as in Theorem \ref{thm:asympbeh}.
We denote by $k$ an arbitrary positive constant that may change from step to step.
By Proposition \ref{prop:asympbehrho}, we have for $x\leq x_1$,
\begin{align*}
	& w\ast\rho(x)  \\
	& \leq k \int_{-\infty}^x w(x-y) \rho(x) dy + k \int_x^{x_1}w(x-y) e^{\frac{\tilde{\delta}_1(c)+\epsilon}{\sigma}(y-x)} \rho(x) dy\\
	& \qquad {} + \int_{x_1}^{x_2} w(x-y)  dy \max_{x_1\leq y \leq x_2} \rho(y) + k \int_{x_2}^{\infty} w(x-y) e^{\frac{\tilde{\delta}_2(c)+\epsilon}{\sigma}(y-x_2)} \rho(x_2) dy\\
	& =: I_1+I_2+I_3+I_4.
\end{align*}
Clearly, $I_1 \leq k \rho(x)$. Since $\tilde{\delta}_1(c)+\epsilon < 1$, also $I_2 \leq k\rho(x)$. Note that $\rho(x) \geq k e^{-\frac{\tilde{\delta}_1(c)+ \epsilon}{ \sigma}(x_1-x)} \rho(x_1)$. As $\int_{x_1}^{x_2} w(x-y) dy = \frac{1}{2}(e^{-\frac{x_1-x}{\sigma}}-e^{-\frac{x_2-x}{\sigma}})$, it follows that 
\[I_3 \leq k e^{-\frac{1-\tilde{\delta}_1(c)-\epsilon}{\sigma}(x_1-x)} \frac{\rho(x)}{\rho(x_1)} \leq k \rho(x).\]
Since $\int_{x_2}^{\infty} w(x-y) e^{\frac{\tilde{\delta}_2(c)+\epsilon}{\sigma}(y-x_2)} dy \leq k e^{\frac{x}{\sigma}}$, we have 
\[I_4 \leq k e^{\frac{x}{\sigma}} e^{\frac{\tilde{\delta}_1(c)+\epsilon}{\sigma}(x_1-x)} \frac{\rho(x)}{\rho(x_1)} \leq k e^{\frac{x_1}{\sigma}} e^{-\frac{1-\tilde{\delta}_1(c)-\epsilon}{\sigma}(x_1-x)} \frac{\rho(x)}{\rho(x_1)} \leq k\rho(x).\]
For $x_1\leq x\leq x_2$, we obtain as above that
\begin{align*}
	w\ast\rho(x)
	& \leq k \rho(x_1) + \max_{x_1\leq y\leq x_2}\rho(y)  + k \rho(x_2) \leq k \frac{\rho(x)}{\min_{x_1\leq y\leq x_2}\rho(y)}.
\end{align*}
Finally, for $x\geq x_2$,
\begin{align*}
	 w\ast\rho(x)
	& \leq k \rho(x_1) + \max_{x_1\leq y\leq x_2}\rho(y) + k \int_{x_2}^x w(x,y) dy \rho(x)\\
	& \qquad {} + k \int_x^{\infty} w(x,y) e^{\frac{\tilde{\delta}_2(c)+\epsilon}{\sigma}(y-x)}\rho(x) dy.
\end{align*}
Noting that $\tilde{\delta}_2(c) +\epsilon<1$ and that $\rho(x) \geq k \rho(x_2)$, we see that also in this case $w\ast\rho(x) \leq k \rho(x)$, which concludes the proof.
\qquad \end{proof}

%% file: multiscale_analysis.bbl
\begin{thebibliography}{10}

\bibitem{amari}
{\sc S.~Amari}, {\em Dynamics of pattern formation in lateral-inhibition type
  neural fields}, Biol. Cybernet., 27 (1977), pp.~77--87.

\bibitem{bloemkerquadratic}
{\sc D.~Blömker and W.~W. Mohammed}, {\em Amplitude equations for spde with
  quadratic nonlinearities}, Electron. J. Probab., 14 (2009), pp.~2527--2550.

\bibitem{bloemkercubic}
\leavevmode\vrule height 2pt depth -1.6pt width 23pt, {\em Amplitude equations
  for spdes with cubic nonlinearities}, Stochastics, 85 (2013), pp.~181--215.

\bibitem{bressloffreview}
{\sc P.~C. Bressloff}, {\em Spatiotemporal dynamics of continuum neural
  fields}, J. Phys. A, 45 (2011), p.~033001.

\bibitem{bressloffbook}
\leavevmode\vrule height 2pt depth -1.6pt width 23pt, {\em Waves in Neural
  Media}, Springer, New York, 2014.

\bibitem{bressloffkilpatrick}
{\sc P.~C. Bressloff and Z.~P. Kilpatrick}, {\em Nonlinear langevin equations
  for the wandering of fronts in stochastic neural fields}, SIAM J. Appl. Dyn.
  Syst.,  (2015), pp.~305--334.

\bibitem{bressloffwebber}
{\sc P.~C. Bressloff and M.~A. Webber}, {\em Front propagation in stochastic
  neural fields}, SIAM J. Appl. Dyn. Syst., 11 (2012), pp.~708--740.

\bibitem{chenfengxin}
{\sc F.~Chen}, {\em Travelling waves for a neural network}, Electron. J.
  Differential Equations, 2003 (2003), pp.~1--4.

\bibitem{coombeswaves}
{\sc S.~Coombes}, {\em Waves, bumps, and patterns in neural field theories},
  Biol. Cybernet., 93 (2005), pp.~91--108.

\bibitem{coombesbook}
{\sc S.~Coombes, P.~beim Graben, R.~Potthast, and J.~Wright}, {\em Neural
  Fields - Theory and Applications}, Springer, Berlin Heidelberg, 2014.

\bibitem{ermentroutreview}
{\sc G.~B. Ermentrout}, {\em Neural networks as spatio-temporal pattern-forming
  systems}, Rep. Progr. Phys., 61 (1998), pp.~353--430.

\bibitem{ermentroutmcleod}
{\sc G.~B. Ermentrout and J.~B. McLeod}, {\em Existence and uniqueness of
  travelling waves for a neural network}, Proc. Roy. Soc. Edinburgh, 123A
  (1993), pp.~461--478.

\bibitem{ermentroutbook}
{\sc G.~B. Ermentrout and D.~H. Terman}, {\em Mathematical Foundations of
  Neuroscience}, Springer, New York, 2010.

\bibitem{faugerasinglis}
{\sc O.~Faugeras and J.~Inglis}, {\em Stochastic neural field equations: A
  rigorous footing}, J. Math. Biol., 71 (2015), pp.~259--300.

\bibitem{inglismaclaurin}
{\sc J.~Inglis and J.~MacLaurin}, {\em A general framework for stochastic
  traveling waves and patterns, with application to neural field equations},
  arXiv:1506.08644,  (2015).

\bibitem{kruegerstannat}
{\sc J.~Krüger and W.~Stannat}, {\em Front propagation in stochastic neural
  fields: A rigorous mathematical framework}, SIAM J. Appl. Dyn. Syst., 13
  (2014), pp.~1293--1310.

\bibitem{lordthuemmler}
{\sc G.~Lord and V.~Thümmler}, {\em Computing stochastic traveling waves},
  SIAM J. Sci. Comput., 34 (2012), pp.~B24--B43.

\bibitem{daprato}
{\sc G.~D. Prato and J.~Zabczyk}, {\em Stochastic Equations in Infinite
  Dimensions}, Cambridge University Press, 1992.

\bibitem{prevotroeckner}
{\sc C.~Pr\'{e}v\^{o}t and M.~Röckner}, {\em A Concise Course on Stochastic
  Partial Differential Equations}, Springer, Berlin, Heidelberg, 2007.

\bibitem{stannatnagumo}
{\sc W.~Stannat}, {\em Stability of travelling waves in stochastic nagumo
  equations}, preprint, arXiv:1301.6378 (2013).

\bibitem{stannatreactiondiffusion}
\leavevmode\vrule height 2pt depth -1.6pt width 23pt, {\em Stability of
  travelling waves in stochastic bistable reaction-diffusion equations},
  preprint, arXiv:1404.3853 (2014).

\bibitem{wilsoncowan72}
{\sc H.~R. Wilson and J.~D. Cowan}, {\em Excitatory and inhibitory interactions
  in localized populations of model neurons}, Biophysical Journal, 12 (1972).

\bibitem{wilsoncowan73}
\leavevmode\vrule height 2pt depth -1.6pt width 23pt, {\em A mathematical
  theory of the functional dynamics of cortical and thalamic nervous tissue},
  Kybernetik, 13 (1973), pp.~55--80.

\end{thebibliography}
